\documentclass[reqno]{amsart}

\usepackage[T1]{fontenc}
\usepackage[utf8]{inputenc}

\usepackage{amsmath}
\usepackage{amsthm}
\usepackage{amssymb}
\usepackage{amsfonts}
\usepackage[shortlabels]{enumitem}
\usepackage[bookmarks=true,bookmarksopen=false,hyperindex,pdftex,colorlinks,citecolor=red, linkcolor=blue]{hyperref}
\usepackage{centernot} 
\usepackage[left=3cm,right=3cm,top=2.5cm,bottom=2.5cm]{geometry}
\usepackage{bbm}
\usepackage[all]{xy}
\usepackage{tikz}
\usetikzlibrary{arrows}
\usetikzlibrary{shapes,decorations}
\usetikzlibrary{positioning}
\usepackage{enumitem}

\usepackage{comment}
\usepackage{centernot}

\DeclareMathOperator{\supp}{supp}                           
\newcommand{\N}{\mathbb{N}}             
\newcommand{\Z}{\mathbb{Z}}             
\newcommand{\C}{\mathbb{C}}             
\newcommand{\K}{\mathbb{K}}             

\newcommand{\Orb}{\mathrm{Orb}} 
 
\newcommand{\Chi}{\mathrm{Chi}}                               
\newcommand{\p}{\mathrm{par}}                               
\newcommand{\ro}{\texttt{r}}                              

\def\<{\langle}
\def\>{\rangle}

\theoremstyle{plain}
\newtheorem{theorem}{Theorem}[section]
\newtheorem{lemma}[theorem]{Lemma}
\newtheorem{corollary}[theorem]{Corollary}
\newtheorem{proposition}[theorem]{Proposition}

\theoremstyle{definition}
\newtheorem*{definition*}{Definition}
\newtheorem{definition}[theorem]{Definition}
\newtheorem{example}[theorem]{Example}

\theoremstyle{remark}
\newtheorem{remark}[theorem]{Remark}

\begin{document}

\title[Backward shift on  directed trees]{On several dynamical properties of Shifts acting on directed trees}

\author[E. Abakumov]{Evgeny Abakumov}

\author[A. Abbar]{Arafat Abbar}

\address[E. Abakumov]{Univ Gustave Eiffel, Univ Paris Est Creteil, CNRS, LAMA UMR8050 F-77447 Marne-la-Vallée, France}
\email{evgueni.abakoumov@univ-eiffel.fr}

\address[A. Abbar]{Univ Gustave Eiffel, Univ Paris Est Creteil, CNRS, LAMA UMR8050 F-77447 Marne-la-Vallée, France}
\email{abbar.arafat@gmail.com}

\date{} 


\begin{abstract}
This paper explores the notions of $\mathcal{F}$-transitivity and topological $\mathcal{F}$-recurrence for backward shift operators on weighted $\ell^p$-spaces and $c_0$-spaces on directed trees, where $\mathcal{F}$ represents a Furstenberg family of subsets of $\mathbb{N}_0$. In particular, we establish the equivalence between recurrence and hypercyclicity of these operators on unrooted directed trees. For rooted directed trees, a backward shift operator is hypercyclic if and only if it possesses an orbit of a bounded subset that is weakly dense.
\end{abstract}


\maketitle

\setcounter{tocdepth}{1}


\numberwithin{equation}{section}

\section{Introduction}
The study of the dynamical properties of weighted shift operators on sequence spaces of trees has been an active research topic in recent years. Martínez-Avendaño \cite{Ma}, Rivera-Guasco and Martínez-Avendaño \cite{RiMa}, and Grosse-Erdmann and Papathanasiou \cite{GrPa, GrPa2} have made significant contributions in this field. See also \cite{AnAv, BaLi} for some further related research. Jblónski, Jung, and Stochel \cite{JJS} introduced and studied weighted forward shift operators on trees, while backward shift operators were introduced in \cite{Ma}.\ In \cite[Theorems 4.3 and 5.2]{GrPa}, the authors provided a complete characterization of the hypercyclicity of backward shift operators on weighted $\ell^p$-spaces and $c_0$-spaces on directed trees. 

Recall that a bounded linear operator is \textit{hypercyclic} if it has a dense orbit. More precisely, let us consider  $X$ to be a Banach space and let $\mathcal{L}(X)$ be the space of bounded linear operators from $X$ into itself. An operator $T\in\mathcal{L}(X)$  is said to be \textit{hypercyclic} if there exists a vector $x\in X$ such that its orbit under $T$, denoted by  $$\Orb(x,T):=\{T^nx:\, n\in\N_0\},$$
 is dense in $X$, where $\N_0=\{0,1,2,\ldots\}$ and $\N=\N_0\setminus\{0\}$. Hypercyclicity is a well-studied concept in Linear Dynamics, see \cite{BaMa, GrPe}. In \cite{Sa1}, Salas provided a description of hypercyclic weighted backward shift operators on $\ell^p$-spaces, which was later extended by Grosse-Erdmann to backward shift operators defined on Fréchet sequence spaces \cite{Gr}. Additionally, Chan and Seceleanu proved that the hypercyclicity of backward shift operators is equivalent to the existence of an orbit with a nonzero limit point \cite{CS}. Some recent improvements to these results have been achieved in \cite{AA, BG, HHY}. An immediate consequence of Chan and Seceleanu's result is that the notions of recurrence and hypercyclicity coincide for backward shift operators on weighted $\ell^p$-spaces of sequences indexed by $\Z$. A vector $x\in X$ is called  \textit{recurrent vector} for an operator $T\in\mathcal{L}(X)$ if there exists a strictly increasing sequence of integers $(n_k)_{k\in\N_0}$ such that 
$$T^{n_k}x\underset{k\to+\infty}{\longrightarrow}x.$$
The operator $T$ is called \textit{recurrent} if its set of recurrent vectors is dense in $X$, or equivalently (see \cite[Proposition 2.1]{CMP}), for every non-empty open set $U$ of $X$, there exists an integer $n\in\N$ such that 
$$T^n(U)\cap U\neq\emptyset.$$
The study of recurrence for linear operators was initiated by Costakis and Parissis \cite{CP}, and later developed further by Costakis, Manoussos, and Parissis \cite{CMP}. For recent contributions, see also \cite{ABBM, BGLP, GrLo, GLP}.
For any nonempty open subsets $U$ and $V$ of $X$, the return set (or, time set) from $U$ to $V$ will be denoted as 
$$N_T(U,V):=N(U,V)=\{n\in\N_0:\, T^n(U)\cap V\neq\emptyset\},$$
When there is no ambiguity,  the index $T$ will be omitted.
Several topological notions in Linear Dynamics can be expressed by using return sets. The operator $T$ is said to be topologically transitive if for any nonempty open subsets $U,V$  of $X$, the return set $N(U,V)$ is nonempty (or, equivalently, infinite). Birkhoff’s transitivity theorem says that hypercyclicity and transitivity coincide, see \cite[Theorem 1.2]{BaMa}. Moreover, $T$ is said to be topologically mixing (resp., weakly mixing, ergodic) if for any nonempty open subsets $U,V$  of $X$, the return set $N(U,V)$ is 
cofinite (resp., thick, syndetic).
Recall that a subset $A$ of $\N_0$ is called 
\begin{itemize}
    \item  \textit{thick}  if it contains arbitrarily long intervals, i.e., for every $n\in\N_0$, there exists $m\in A$ such that $[\![m,m+n]\!]\subset A$. 
    \item  \textit{syndetic}  if it has bounded gaps, i.e., there exists an integer $n\in\N$ such that for all $m\in\N_0$, we have $[\![m,m+n]\!]\cap A\neq\emptyset$.
\end{itemize}
Furstenberg families, as defined in subsection \ref{Fur_families}, serve to generalize the concepts of topological transitivity, mixing, weakly mixing, and ergodicity in terms of $\mathcal{F}$-transitivity, see \cite{BMPP2}.  Let $\mathcal{F}$  be a Furstenberg family of subsets of $\N_0$.
 The operator $T$ is called \textit{$\mathcal{F}$-transitive} if, for any pair of nonempty open subsets $U$ and $V$ in $X$, the set $N(U,V)$ belongs to $\mathcal{F}$. Similarly, following a similar approach (see \cite[Remark 8.2]{BGLP}), $T$ is said to be topologically $\mathcal{F}$-recurrent if, for any nonempty open subset $U$ in $X$, the set $N(U,U)$ belongs to $\mathcal{F}$. 
In \cite{BMPP2}, the authors provided a characterization of $\mathcal{F}$-transitive weighted bilateral backward shift operators on the spaces $\ell^p(\mathbb{Z})$ and $c_0(\mathbb{Z})$, as well as $\mathcal{F}$-transitive weighted unilateral backward shift operators.

\smallskip

This paper is devoted to studying $\mathcal{F}$-transitive and topological $\mathcal{F}$-recurrence backward shift operators on trees, which extends the results obtained in  \cite{BMPP2, GrPa}. Additionally, we investigate the concept of $\Gamma$-supercyclicity, introduced in \cite{CEM}, for these operators. The structure of this paper is outlined as follows:

In section \ref{Pre}, we introduce the necessary notations and definitions for our analysis, including  Furstenberg families and a review of backward shift operators on weighted $\ell^p$-spaces and $c_0$-spaces on trees.  

In section  \ref{section3}, we give a variant of the $\mathcal{F}$-transitivity criterion that we will specifically use to describe $\mathcal{F}$-transitive backward shift operators on unrooted directed trees (see Theorem \ref{Trans_Crite}).

Section \ref{section4} focuses on characterizing $\mathcal{F}$-transitive backward shift operators on weighted $\ell^p$-spaces and $c_0$-spaces defined on rooted directed trees. We demonstrate that $\mathcal{F}$-transitivity is equivalent to a property weaker than topological $\mathcal{F}$-recurrence (see Theorem \ref{FTransitive_rooted case}). Additionally, we derive a corollary indicating that hypercyclicity of a backward shift operator $B$ on a weighted $\ell^p$-space or $c_0$-space defined on a rooted directed tree is equivalent to the existence of a bounded subset whose orbit under $B$ is weakly dense in the underlying space (see  Corollary \ref{hc_rooted}).

In section \ref{section5}, we provide a characterization of $\mathcal{F}$-transitivity and establish its equivalence with topological $\mathcal{F}$-recurrence for backward shift operators on weighted $\ell^p$-spaces and $c_0$-spaces defined on unrooted directed trees (see Theorem \ref{FTransitive_unrooted case}).

In section \ref{section6}, we discuss the concept of $\Gamma$-supercyclicity for  backward shift operators on trees (see Theorem \ref{super_criterion}). This enables us to recover the characterization of hypercyclicity for these operators, as established in \cite{GrPa}, and also to provide a characterization of their supercyclicity.

Finally, in the last section, we demonstrate that unlike the cases of the trees $\mathbb{N}_0$ or $\mathbb{Z}$, there exist backward shift operators on weighted spaces on trees that are not hypercyclic but possess orbits with a non-zero limit point (see Examples \ref{rooted_not_hc} and \ref{example1}). Theorem \ref{recu_tree_shift} characterizes backward shifts that have an orbit with a non-zero limit point in the case of rooted trees. For unrooted cases, we provide a characterization of backward shifts that have an orbit of some non-negative function with a non-zero limit point (see Proposition \ref{th_unrooted_tree}).

\section{Preliminaries}\label{Pre}
\subsection{Directed trees}\label{trees} In what follows, we will recall the needed terminologies related to trees, for more information we refer to \cite{JJS,GrPa} and the references therein.
A \textit{(directed) graph} is a pair $\mathcal{G}=(V,E)$ which satisfies the two conditions:
\begin{enumerate}
    \item $V$ is a non-empty set, its elements are called \textit{vertices}.
    \item $E$ is a subset of $\lbrace (u,v): u,v\in V, u\neq v\rbrace$, its elements are called \textit{edges}.
\end{enumerate}
An \textit{undirected edge} of $\mathcal{G}$  means an element of the set
$$\tilde{E}:=\lbrace \lbrace u,v\rbrace:\, (u,v)\in E \text{ or } (v,u)\in E\rbrace.$$
If $\lbrace u,v\rbrace\in \tilde{E}$, we say that $u$ and $v$ are adjacent and we denote that by $u\backsim v$.
\begin{itemize}
    \item The  graph $\mathcal{G}$ is \textit{connected} if any two distinct vertices $u\neq v\in V$  are related by an  \textit{undirected path}, i.e., there exist $v_1,\ldots,v_n\in V$  such that $u=v_1\backsim v_2\backsim \cdots\backsim v_n=v$.
 \item A \textit{circuit} of $\mathcal{G}$ is a  sequence $\lbrace v_j\rbrace_{j=1}^{n}$ of distinct vertices ($n\geqslant2$) such that, for all $j\in\{ 1,\ldots,n-1\}$,
$$(v_j,v_{j+1})\in E \quad \text{and} \quad (v_n,v_1)\in E.$$
 
\end{itemize}

\begin{definition}[Directed trees]
A \textit{directed tree} $\mathcal{T}=(V,E)$  is a  directed graph such that
\begin{enumerate}
\item $(V,E)$ is connected.
\item The set $E$ of edges has no circuits. 
\item For any vertex $v\in E$, there exists at most one vertex $u\in V$ such that $(u,v)\in E$.
\item The set $V$ of vertices is countable.
\end{enumerate}
\end{definition}
In what follows, let us assume that  $(V,E)$ is a directed tree.
\begin{itemize}
\item A \textit{parent} of a vertex $v\in V$ is a vertex $u\in V$ such that $(u,v)\in E$, we denote such vertex by $\p(v)$.
\item A \textit{child} of a vertex $v\in V$ is a vertex $u\in V$ whose parent is $v$. We denote by $\Chi(v)$ the set of children of $v$, that is, 
$$\Chi(v)=\lbrace u\in V: \, \p(u)=v\rbrace.$$
Moreover, for every $n\geqslant2$, we denote
$$\Chi^n(v):=\bigcup_{u\in \Chi(v)}\Chi^{n-1}(u),$$
where $\Chi^{1}(v)=\Chi(v)$ and $\Chi^{0}(v)=\lbrace v\rbrace$.
\item[•] A \textit{root} of $(V,E)$ is a  vertex without a parent.   Any directed tree has at most one root, if such an element exists, it will be denoted simply by $\ro$. 
\item A \textit{leaf} of $(V,E)$ is a vertex without children. 
\end{itemize}

\subsection{Backward shift on sequences spaces on directed trees.}\label{sub_shift}
Let $(V,E)$ be a directed tree, let $\mu=(\mu_v)_{v\in V}$ be a \textit{weight} on $V$, that is, a sequence of non-zero numbers and let $\K$ denote the set of real or complex numbers. For $1\leqslant p<+\infty$,  the weighted $\ell^p$-space of $V$ is defined by
$$\ell^p(V,\mu):=\Big\lbrace f\in\mathbb{K}^{V}: \,\sum_{v\in V}|f(v)\,\mu_v|^p<+\infty\Big\rbrace,$$
quipped with the norm
$$\|f\|_{p,\mu}:=\Big(\sum_{v\in V}|f(v)\,\mu_v|^p\Big)^{1/p}$$
is a Banach space. 
As usually, $\ell^\infty(V,\mu)$ is the Banach space of function $f\in\mathbb{K}^{V}$ such that
$$\|f\|_{\infty,\mu}=\underset{v\in V}{\sup}|f(v)\,\mu_v|<+\infty.$$
Finally,  the weighted $c_0$-space on $V$ is defined by
$$c_0(V,\mu):=\Big\lbrace f\in\mathbb{K}^{V}: \, \forall\varepsilon>0, \exists F\subset V \text{  finite}, \forall v\in V\setminus F, |f(v)\,\mu_v|<\varepsilon\Big\rbrace,$$
endowed with the norm $\|.\|_{\infty,\mu}$ is a closed subspace of $\ell^\infty(V,\mu)$. 
For $v\in V$, let $e_v=\chi_{\lbrace v\rbrace}$ be the characteristic function of $\lbrace v\rbrace$. Note that the space $\mathrm{span}\lbrace e_v:\, v\in V\rbrace$
is dense in $X=\ell^p(V,\mu)$, $1 \leq p<+\infty$ or $X=c_0(V,\mu)$. The support of a function $f\in \K^{V}$ is the set 
$$\supp(f):=\lbrace v\in V:\, f(v)\neq0\rbrace.$$
According to \cite{Ma}, the backward shift $B$ on $\K^V$ is defined  by:
$$(Bf)(v)=\sum_{u\in \mathrm{Chi}(v)} f(u),\quad v\in V,$$
where an empty sum is zero. It  can be seen as the adjoint of the forward shift operator $S$ on trees, which is naturally defined on $\K^V$ as follows, for any $f\in \K^V$ and $v\in V$:
\begin{center}
\begin{minipage}{0.2\linewidth}
~~
\end{minipage}
\begin{minipage}{0.4\linewidth}
\begin{equation*}
(Sf)(v)=\begin{cases}
f(\p(v))&\text{ if } v\neq \ro \\
&\\
0 &\text{ if } v=\ro
\end{cases},
\end{equation*}
\end{minipage}
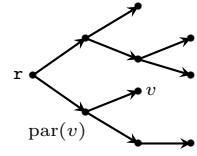
\begin{minipage}{0.08\linewidth}
~~
\end{minipage}
\begin{minipage}[c]{0.3\linewidth}
\begin{tikzpicture}[thick, scale=0.7]

  \draw (0,0) node {{\tiny $\bullet$}} ;
\draw (1,0.7) node {{\tiny $\bullet$}} ;
\draw (1,-0.7) node {{\tiny $\bullet$}} ;
\draw (2,0.3) node {{\tiny $\bullet$}} ;
\draw (2,1.3) node {{\tiny $\bullet$}};
\draw (2,-0.3) node {{\tiny $\bullet$}};
\draw (2,-1.3) node {{\tiny $\bullet$}};

\draw (3,0) node {{\tiny $\bullet$}};
\draw (3,0.7) node {{\tiny $\bullet$}};
\draw (3,-1.3) node {{\tiny $\bullet$}};

\tikzstyle{suite}=[->,>=stealth,thick,rounded corners=4pt]

\draw[suite] (0,0) -- (1,0.7) edge (2,1.3);
\draw[suite] (1,0.7) --(2,0.3)  edge (3,0.7);
\draw[suite] (0,0) -- (1,-0.7) edge (2,-0.3);
\draw[suite] (1,-0.7) -- (2,-1.3) edge (3,-1.3);
\draw[suite] (2,0.3) -- (3,0);

\node[left=2pt] at (0.1,0) {{\scriptsize $\ro$}};
\node[right] at (1.95,-0.3) {{\scriptsize $v$}};
\node[below left] at (1.2,-0.7) {{\scriptsize $\p(v)$}};
\end{tikzpicture}
\end{minipage}
\end{center}
Under the following pairing of duality
$$\langle f,g \rangle=\sum_{v\in V}f(v)\,g(v),$$
the dual of $\ell^p(V,\mu)$, for $1\leq p<+\infty$, is $\ell^{p^\ast}(V,1/\mu)$, where $p^\ast$ is the conjugate exponent of $p$, and the dual of $c_0(V,\mu)$ is $\ell^1(V,1/\mu)$. Now, if $B$ or $S$ defines a bounded linear operator on $X=\ell^p(V,\mu)$, $1\leq p <+\infty$, or $X=c_0(V,\mu)$ then 
$$\langle Bf,g \rangle=\langle f,Sg \rangle,$$
and this is why $B$ has been named the backward shift. Alternatively, we can define the backward shift $B$ as the unique bounded linear operator that satisfies
$$
Be_v=\begin{cases}
e_{\p(v)}&\text{ if } v\neq \ro \\
0 &\text{ if } v=\ro
\end{cases},\qquad v\in V.
$$
In the following proposition, we will recall the necessary and sufficient conditions for the boundedness of the backward shift on weighted $\ell^p$-spaces or $c_0$-spaces on trees, see \cite[Proposition 2.3]{GrPa}. 
\begin{proposition} \label{Bounded_Back}
Let $(V,E)$ be a directed tree, let $\mu=(\mu_v)_{v\in V}$ be a weight on $V$, and let $B$ be the backward shift on $\K^V$.
\begin{itemize}
\item[$(a)$] $B$ is a bounded linear operator on $\ell^1(V,\mu)$ if and only if 
$$\sup_{v\in V\setminus\lbrace\ro\rbrace}\Big|\frac{\mu_{\p(v)}}{\mu_v}\Big|<+\infty.$$
In this case, $\|B\|=\underset{v\in V\setminus\lbrace\ro\rbrace}{\sup} \big|\frac{\mu_{\p(v)}}{\mu_v}\big|$.

\item[$(b)$] Let $1<p<+\infty$. $B$ is a bounded linear operator on $\ell^p(V,\mu)$ if and only if 
$$\underset{v\in V}{\sup}\, \sum_{u\in \Chi(v)} \Big|\frac{\mu_v}{\mu_u  }\Big|^{p^\ast}<+\infty.$$
In this case, $\|B\|=\underset{v\in V}{\sup} \, \Big(\underset{{u\in \Chi(v)}}{\sum} \big|\frac{\mu_v}{\mu_u }\big|^{p^\ast}\Big)^{1/p^{\ast}}$.

\item[$(c)$] $B$ is a bounded linear operator on $c_0(V,\mu)$ if and only if 
$$\underset{v\in V}{\sup}\, \sum_{u\in \Chi(v)} \Big|\frac{\mu_v}{\mu_u  }\Big|<+\infty.$$
In this case, $\|B\|=\underset{v\in V}{\sup}\, \underset{u\in \Chi(v)}{\sum} \big|\frac{\mu_v}{\mu_u}\big|$.
\end{itemize}
\end{proposition}

 As mentioned previously, K. Grosse-Erdmann and D. Papathanasiou provided a comprehensive characterization of hypercyclicity for backward shift operators on weighted $\ell^p$-spaces and $c_0$-spaces. In the cases of rooted directed trees,  they obtained the following theorem, see \cite[Theorem 4.3]{GrPa}.
\begin{theorem}
 Let $(V, E)$ be a rooted directed tree and let $\mu=(\mu_v)_{v\in V}$ be a weight on $V$. Let $X=\ell^p(V,\mu)$, $1\leq p <+\infty$, or $X=c_0(V,\mu)$ and suppose  that the  backward shift $B$ is a bounded operator on $X$.  Then the following assertions are equivalent:
 \begin{enumerate}[label={$(\arabic*)$}]
\item $B$ is hypercyclic.
\item $B$ is weakly mixing.
 \item  There is an increasing sequence $(n_k)_{k\in\N}$ of positive integers such that, for each $v \in V$, we have 
 $$\begin{cases}
\underset{u \in \Chi^{n_k}(v)}{\sum}\dfrac{1}{|\mu_u|^{p^{\ast}}}\underset{k\to+\infty}{\longrightarrow}+\infty & \text{ if } X=\ell^p(V,\mu),\, 1< p<+\infty\\
 &\\
\underset{u\in\Chi^{n_{k}}(v)}{\inf}\,|\mu_u|\underset{k\to+\infty}{\longrightarrow}0 &\text{ if } X=\ell^1(V,\mu), \\
&\\
\underset{u \in \Chi^{n_k}(v)}{\sum}\dfrac{1}{|\mu_u|}\underset{k\to+\infty}{\longrightarrow}+\infty&\text{ if } X=c_0(V,\mu),
\end{cases}.$$
\end{enumerate}
\end{theorem}
For the case of unrooted directed trees, we refer to \cite[Theorem 5.2]{GrPa}. 
In what follows, we will require the following lemma, which played a crucial role in establishing the results in \cite{GrPa}.

\begin{lemma}[\protect{\cite[Lemma 4.2]{GrPa}}]\label{Lemma_Grosse}
Let $J$ be a finite or countable set and let $\mu=(\mu_j)_{j\in J}\in(\K\setminus\lbrace0\rbrace)^{J}$. Then 
$$\underset{\|x\|_1=1}{\inf}\sum_{j\in J}|x_j\mu_j|=\underset{j\in J}{\inf}\,|\mu_j|,$$
$$\underset{\|x\|_1=1}{\inf}\Big(\sum_{j\in J}|x_j\mu_j|^p\Big)^{1/p}=\Big(\underset{j\in J}{\sum}\,\frac{1}{|\mu_j|^{p^\ast}}\Big)^{-1/p^{\ast}},\quad 1<p<+\infty, \text{ and } p^{\ast}=\frac{p}{p-1},$$
$$\underset{\|x\|_1=1}{\inf}\sup_{j\in J}|x_j\mu_j|=\Big(\underset{j\in J}{\sum}\,\frac{1}{|\mu_j|}\Big)^{-1},$$
where $x\in \K^J$, $\|x\|_{1}=\sum_{j\in J}|x_j|$ and $\infty^{-1}=0$. The same holds when the sequences $x$ are required, in addition, to be of finite support.
\end{lemma}

\subsection{Furstenberg families} \label{Fur_families} We devote this subsection to recalling the necessary definitions related to Furstenberg families. A non-empty family $\mathcal{F}$ of subsets of $\N_0$ is a \textit{Furstenberg family}, if for all $A \in \mathcal{F}$ it holds
\begin{itemize}
    \item $A$ is infinite; 
    \item  If $A \subset B \subset \N_0$, then $B\in\mathcal{F}$.
\end{itemize}
Let $\mathcal{F}\subset\mathcal{P}(\N_0)$ be a Furstenberg family. Following \cite{BMPP2}, we denote by $\mathcal{\widetilde{F}}$ (resp., $\mathcal{\widetilde{F}}_{+}$) the Furstenberg family consisting of subsets $A\subset \N_0$ such that  for every $N\in\N_0$, there exists  $B\in\mathcal{F}$  satisfying  
$$(B+[\![-N, N ]\!])\cap\N_0 \subset A \quad \text{(resp., } B+[\![0, N ]\!] \subset A).$$
It is clear that $\mathcal{\widetilde{F}}\subset \mathcal{\widetilde{F}}_{+}\subset \mathcal{F}$; thus, any $\mathcal{\widetilde{F}}$-transitive operator is $\mathcal{F}$-transitive. 
\smallskip

 A {\it filter} on $\mathbb{N}_0$ is a Furstenberg family $\mathcal{F}$ of subsets of $\mathbb{N}_0$ that satisfies the property $A \cap B \in \mathcal{F}$ whenever $A, B \in \mathcal{F}$. A nonempty collection $\mathcal{B}$ of subsets of $\mathbb{N}_0$ is called a {\it filter base} if the following conditions hold:
\begin{itemize}
\item Every set in $\mathcal{B}$ is infinite.
\item The intersection of any two sets in $\mathcal{B}$ contains a set in $\mathcal{B}$.
\end{itemize}
It is worth noting that every filter is a filter base. Conversely, if $\mathcal{B}$ is a filter base of subsets of $\mathbb{N}_0$, then the collection of sets
$$\mathcal{F}_{\mathcal{B}}=\{A\subset \N_0: B\subset  A \text{ for some } B\in \mathcal{B}\}$$
forms a filter, called the filter generated by $\mathcal{B}$.

\section{$\mathcal{F}$-Transitivity Criterion} \label{section3}

 In \cite[Theorem 2.4]{BMPP2}, Bès, Menet, Peris, and Puig established an $\mathcal{F}$-Transitivity Criterion. This criterion uses the concept of a limit along the family $\mathcal{F}$. Let $(x_n)_n$ be a sequence in $X$, and let $x \in X$. We say that
$$\mathcal{F}\text{-}\underset{n}{\lim}\,x_n=x,$$
if $\{n\in\N_0:\, x_n\in U\}\in\mathcal{F}$ holds for every neighborhood $U$ of $x$. In our analysis, we will require a variant of this criterion, combined with the Hypercyclicity Criterion \cite[Proposition 5.1]{GrPa}. For the sake of completeness, we will include its proof, which is an adaptation of the proof provided in \cite[Theorem 2.4]{BMPP2}.
\begin{theorem}[$\mathcal{F}$-Transitivity Criterion]\label{Trans_Crite}
Let $X$ be an infinite-dimensional Banach space, $T\in \mathcal{L}(X)$ and let $\mathcal{F}$ be a Furstenberg family on $\N_0$ such that $\widetilde{\mathcal{F}}$ is a filter. Suppose that there exist two dense subsets $X_0$ and $Y_0$ in $X$ and maps $I_n:X_0\to X$ and   $S_n: Y_0\to X$, $n\in \N_0$, such that, for any $x\in X_0$ and any $y\in Y_0$, the following conditions hold:
\begin{enumerate}[label=$(\arabic*)$]
    \item $\mathcal{F}\text{-}\underset{n}{\lim}\,(I_nx,T^nI_nx)=(x,0)$;
    \item $\mathcal{F}\text{-}\underset{n}{\lim}\,(S_ny,T^nS_ny)=(0,y)$.
\end{enumerate}
Then $T$ is $\widetilde{\mathcal{F}}$-transitive. 
\end{theorem}
\begin{proof}
Note that, according to condition $(2)$,   $T$ has a dense range in $X$. Now, let $U,V$ be nonempty open subsets of $X$. There exist nonempty open subsets $U',V'$ of $X$ and a $0$-neighbourhood $W$ such that
$$U'+W\subset U\quad \text{and}\quad V'+W\subset V.$$
Since $\widetilde{\mathcal{F}}$ is a filter and
$$N(U,V) \supset N(U'+W,V'+W)\supset N(U',W)\cap N(W,V'),$$
it is enough to show that $N(U',W)\in \widetilde{\mathcal{F}}$ and $N(W,V')\in \widetilde{\mathcal{F}}$.
Let $N\in\N_0$. Since $T$ has a dense range, the sets $T^{-N}U'$ and $T^{-N}V'$ are nonempty open sets. Choose  $x\in X_0\cap T^{-N}U'$ and $y\in Y_0\cap T^{-N}V'$. We start by showing that $N(T^{-N}U',W)\in \widetilde{\mathcal{F}}_{+}$. For any $M\in\N_0$, by $(1)$, there exists $A_M\in \mathcal{F}$ such that, for every $n\in A_M$,
$$I_nx\in T^{-N}U' \quad \text{and} \quad T^nI_nx\in  \bigcap_{k=0}^{M}T^{-k}W.$$
Note that, for any $n\in A_M+[\![0,M]\!]$, there exists $k\in [\![0,M]\!]$ such that 
$n-k\in A_M$, hence
$$I_{n-k}x \in T^{-N}U' \quad\text{and} \quad T^{n-k}I_{n-k}x\in  T^{-k}W,$$
therefore, $I_{n-k}x \in T^{-N}U'$ and  $T^{n}I_{n-k}x\in  W$, so $n\in N(T^{-N}U',W)$. Thus $A_M+[\![0,M]\!]\subset N(T^{-N}U',W)$ holds for any arbitrarily $M\in\N_0$. Consequently, $N(T^{-N}U',W) \in \widetilde{\mathcal{F}}_{+}$. Thus, there is $B\in\mathcal{F}$ such that $B+[\![0,2N]\!]\subset N(T^{-N}U',W)$. Therefore,
$$(B +[\![-N,N]\!])\cap\N_0 \subset (N(T^{-N}U',W)-N)\cap \N_0 \subset N(U',W),$$
and, since $N$ was arbitrary, we deduce  that $N(U',W)\in \widetilde{\mathcal{F}}$.

 On the other hand, by $(2)$, there is $B\in\mathcal{F}$ such that, for every $n\in B$,
$$S_n y\in \bigcap_{k=0}^{2N}T^{-k}W \quad\text{and}\quad T^n S_n(y)\in T^{-N}V'.$$
In particular, for any $n\in(B+[\![-N,N]\!])\cap \N_0$, there exists  $k\in[\![-N,N]\!]$ such that $n-k\in B$ and $N-k\in[\![0,2N]\!]$, thus 
$$S_{n-k}y\in T^{k-N}W  \quad \text{and} \quad T^{n-k}S_{n-k}y\in T^{-N}V',$$
hence
$$T^{N-k}S_{n-k}y\in W  \quad \text{and} \quad T^{n}(T^{N-k}S_{n-k}y)\in T^{-N}V',$$
 from which we deduce that $n\in N(W,V')$. Therefore, 
$$(B+[\![-N,N]\!])\cap \N_0 \subset N(W,V').$$
Since $N$ was arbitrary, we obtain that $N(W,V')\in\widetilde{\mathcal{F}}$.
This finishes the proof.
\end{proof}

\section{$\mathcal{F}$-transitivity - rooted case} \label{section4} 
Our first main result provides a characterization of $\mathcal{F}$-transitivity for backward shifts on rooted directed trees and establishes its equivalence with certain weak properties.

\begin{theorem}\label{FTransitive_rooted case}
 Let $(V,E)$ be a rooted directed tree, let $\mu=(\mu_v)_{v\in V}$ be a weight on $V$, and let $\mathcal{F}$ be a Furstenberg family on $\N_0$. Let $X=\ell^p(V,\mu)$, $1\leq p <+\infty$, or $X=c_0(V,\mu)$ and suppose  that the  backward shift $B$ is a bounded operator on $X$. The following assertions are equivalent: 
 \begin{enumerate}[label={$(\arabic*)$}]
  \item $B$ is $\widetilde{\mathcal{F}}$-transitive.
 \item $B$ is $\mathcal{F}$-transitive.
 \item There exists a bounded subset $C\subset X\setminus \{0\}$ such that for every nonempty weak open subset $W\subset X$, $N(C,W)\in\mathcal{F}$.
 \item For every $N\in\N$ and every finite subset $F\subset V$, we have
 $$\bigcap_{v\in F}\Big\{n\in\N_0:\, \underset{u\in\Chi^n(v)}{\sup}\dfrac{1}{|\mu_u|}>N\Big\}\in\mathcal{F},\quad \text{if } X=\ell^1(V,\mu);$$
 
  $$\bigcap_{v\in F}\Big\{n\in\N_0:\, \Big(\underset{u\in\Chi^n(v)}{\sum}\dfrac{1}{|\mu_u|^{p^\ast}}\Big)^{1/p^\ast}>N\Big\}\in\mathcal{F},\quad \text{if }X=\ell^p(V,\mu),\, 1<p<+\infty;$$
  
   $$\bigcap_{v\in F}\Big\{n\in\N_0:\, \underset{u\in\Chi^n(v)}{\sum}\dfrac{1}{|\mu_u|}>N\Big\}\in\mathcal{F},\quad \text{if }X=c_0(V,\mu).$$
 \end{enumerate}
If $\mathcal{F}$ satisfies that $A\cap [n,+\infty[ \in\mathcal{F}$ whenever $A\in\mathcal{F}$ and $n\in\N$, then the above conditions are equivalent to
\begin{enumerate}
    \item[$(5)$] For every $g\in X$, for every open neighborhood $U$ of $g$, for every weakly open neighborhood $W$ of $g$, $N(U,W)\in\mathcal{F}$.
\end{enumerate} 
If $\mathcal{F}$ is a filter,  then the conditions $(1)\text{-}(4)$ are equivalent to
\begin{enumerate}
    \item[$(6)$] For every $N\in\N$ and every $v\in V$, we have
 $$\Big\{n\in\N_0:\, \underset{u\in\Chi^n(v)}{\sup}\dfrac{1}{|\mu_u|}>N\Big\}\in\mathcal{F},\quad \text{if } X=\ell^1(V,\mu);$$
 
  $$\Big\{n\in\N_0:\, \Big(\underset{u\in\Chi^n(v)}{\sum}\dfrac{1}{|\mu_u|^{p^\ast}}\Big)^{1/p^\ast}>N\Big\}\in\mathcal{F},\quad \text{if }X=\ell^p(V,\mu),\, 1<p<+\infty;$$
  
   $$\Big\{n\in\N_0:\, \underset{u\in\Chi^n(v)}{\sum}\dfrac{1}{|\mu_u|}>N\Big\}\in\mathcal{F},\quad \text{if }X=c_0(V,\mu).$$
\end{enumerate}
\end{theorem}
\begin{proof}
We will only prove the equivalences in the case where $X = \ell^p(V,\mu)$ with $1 < p < +\infty$. A similar argument can be made to deduce the cases where $X = \ell^1(V,\mu)$ and $X = c_0(V,\mu)$. It is clear that  $(2)\Rightarrow(3)$ and  $(2)\Rightarrow(5)$. Moreover, $(1)\Rightarrow(2)$ since $\widetilde{\mathcal{F}} \subset \mathcal{F}$. Suppose that $B$ is $\mathcal{F}$-transitive, thus it is transitive (equivalently, hypercyclic). Using \cite[Theorem 4.3]{GrPa}, we can conclude that $B$ is weakly mixing. Thus, $B$ is both $\mathcal{F}$-transitive and weakly mixing, which is equivalent to saying that $B$ is $\widetilde{\mathcal{F}}$-transitive, according to \cite[Lemma 2.3]{BMPP2}. Hence, $(2)\Rightarrow(1)$. Moreover, it clear that $(4)\Leftrightarrow (6)$ when $\mathcal{F}$ is a filter.

\smallskip

  $(3)\Rightarrow (4)$.  Let  $C$ be a bounded subset of $\ell^p(V,\mu)\setminus\{0\}$ such that $N(W,C)\in\mathcal{F}$, for every noempty weak open subset $W$ of $\ell^p(V,\mu)$.  Set $M:=\sup\{\|f\|_{p,\mu}:\, f\in C\}<+\infty$. Let $N\in\N$ and let $F\subset V$ be a finite subset. Then
$$W=\{f\in\ell^p(V,\mu):\, |\langle f-(MN+1)\sum_{u\in F}e_u,e_v\rangle|<1,\,\forall v\in F\}$$
is a weakly open neighborhood of $(MN+1)\sum_{u\in F}e_u$. By the hypothesis, $N(C,W)\in\mathcal{F}$. Since $\mathcal{F}$ is a Furstenberg family, we need only to prove  that 
\begin{equation}
  N(C,W)\subset \bigcap_{v\in F}\Big\{n\in\N_0:\, \Big(\underset{u\in\Chi^n(v)}{\sum}\dfrac{1}{|\mu_u|^{p^\ast}}\Big)^{1/p^\ast}>N\Big\}. 
  \label{eqq2}
\end{equation}
Let $n\in N(C,W)$. There exists then $f\in C$ such that $B^{n}f\in W$, thus for every $v\in F$
$$|(B^{n}f)(v)-(MN+1)|<1,$$
hence,  by using  Hölder's inequality, we obtain
\begin{align*}
   MN&<\sum_{u\in\Chi^{n}(v)}|f(u)|\\
   &\leq \Big( \sum_{u\in\Chi^{n}(v)}|f(u)\mu_u|^p\Big)^{1/p}\Big( \sum_{u\in\Chi^{n}(v)}\dfrac{1}{|\mu_u|^{p^\ast}}\Big)^{1/p^\ast}\\
   &\leq \|f\|_{p,\mu}\Big( \sum_{u\in\Chi^{n}(v)}\dfrac{1}{|\mu_u|^{p^\ast}}\Big)^{1/p^\ast}\\
   &\leq M\,\Big( \sum_{u\in\Chi^{n}(v)}\dfrac{1}{|\mu_u|^{p^\ast}}\Big)^{1/p^\ast}\\
\end{align*}
therefore
$$\Big( \sum_{u\in\Chi^{n}(v)}\dfrac{1}{|\mu_u|^{p^\ast}}\Big)^{1/p^\ast}>N,\quad \forall v\in F,$$
and so $\eqref{eqq2}$ holds.

\smallskip

Let us show that $(4)\Rightarrow(2)$. Assume that $(4)$ holds, that is, for every $N>0$ and every finite subset $F\subset V$, we have
   $$I(F,N):= \bigcap_{v\in F}\Big\{n\in\N_0:\, \Big(\underset{u\in\Chi^n(v)}{\sum}\dfrac{1}{|\mu_u|^{p^\ast}}\Big)^{1/p^\ast}>N\Big\}\in\mathcal{F}.$$
   Let $\mathcal{B}$ be the filter base consisting of all subsets $I(F,N)$ of $\N_0$, where  $N>0$ and  $F\subset V$ is finite. We define
   $$\mathcal{F}_{\mathcal{B}}=\{A\subset \N_0: B\subset  A \text{ for some } B\in \mathcal{B}\}.$$ 
   In other words, $\mathcal{F}_\mathcal{B}$ is the filter generated by $\mathcal{B}$. We will show that $B$ satisfies the $\mathcal{F}_{\mathcal{B}}$-transitivity Criterion, which implies that $B$ is $\mathcal{F}$-transitive since $\mathcal{F}_{\mathcal{B}}\subset \mathcal{F}$.
   Set $\mathcal{D}=\mathrm{span}\{e_v:\, v \in V\}$. Note that $\mathcal{D}$ is dense in $\ell^p(V,\mu)$. Let $f=\sum_{u\in F}f(u) e_u\in \mathcal{D}$, with $F\subset V$ finite. Let $\varepsilon>0$ and $\mathcal{U}:=B(0,\varepsilon)$ be the open ball of center $0$ and radius $\varepsilon$ in $\ell^p(V,\mu)$. 
   Let us start by proving that
$$N(f,\mathcal{U}):=\{n\in\N_0:\, B^{n}f\in \mathcal{U}\}\in \mathcal{F}_\mathcal{B}.$$
To achieve this, it is enough to find a number $N\geq1$ such that $I(F,N) \subset N(f,\mathcal{U})$. First, select an integer $n_0\in \mathbb{N}$ large enough such that $\Chi^{n}(V)\cap F=\emptyset$ for every integer $n\geq n_0$. Then, we set
$$N:=1+\underset{v\in F}{\max}\,\underset{0\leq k<n_0}{\max}\Big(\underset{u\in\Chi^k(v)}{\sum}\dfrac{1}{|\mu_u|^{p^\ast}}\Big)^{1/p^\ast}.$$
Note that, by using Proposition \ref{Bounded_Back} and the boundedness of $B^{k}$ for every $0\leq k<n_0$, $N$ is finite. Now, suppose that $n\in I(F,N)$. This implies that $n\geq n_0$. Thus, for every $v\in V$, we have $\Chi^n(v)\cap F=\emptyset$, hence, $(B^{n} f)(v)=0$. This implies that $n\in N(f,\mathcal{U})$. Therefore, we have shown that $I(F,N) \subset N(f,\mathcal{U})$, and consequently, $N(f,\mathcal{U})\in\mathcal{F}_\mathcal{B}$ since $\mathcal{F}_\mathcal{B}$ is a Furstenberg family. Hence, the first condition of the $\mathcal{F}_\mathcal{B}$-transitivity criterion holds, where the maps $I_n$ correspond to the identity map.

Now, for every $v\in V$, let $(\delta_{v,n})_n$ be a decreasing sequence of positive real numbers tending to zero.  For all $v\in V$ and $n\in\N_0$, by Lemma \ref{Lemma_Grosse}, we have
$$\underset{\|x\|_1=1}{\inf}\Big(\underset{u \in \Chi^{n}(v)}{\sum}|x_u\mu_u|^p\Big)^{1/p}=\Big(\underset{u \in \Chi^{n}(v)}{\sum}\,\frac{1}{|\mu_u|^{p^\ast}}\Big)^{-1/p^{\ast}}.$$
There exists then $g_{v,n}\in\K^V$ non-negative, of support in $\Chi^{n}(v)$, such that 
\begin{equation}\label{rooted:eqq1} 
    \sum_{u\in \Chi^{n}(v)}g_{v,n}(u)=1 \quad \text{and} \quad \Big(\sum_{u\in \Chi^{n}(v)}|g_{v,n}(u)\,\mu_u|^p\Big)^{1/p}<\Big(\underset{u \in \Chi^{n}(v)}{\sum}\,\frac{1}{|\mu_u|^{p^\ast}}\Big)^{-1/p^{\ast}}+\delta_{v,n}.
\end{equation}
Note that,  $g_{v,n}\in \ell^p(V,\mu)$. Let us define linearly the maps $S_k:\mathcal{D}\to \ell^p(V,\mu)$, for $k\geq 0$, by
$$S_ke_v=g_{v,k},\quad \forall v\in V.$$
Let $\varepsilon>0$ and $g=\sum_{u\in F} g(a) e_a\in\mathcal{D}$, where $F$ is the support of  $g$. Set again $\mathcal{U}:=B(0,\varepsilon)$ 
and $\mathcal{V}:=B(g,\varepsilon)$. 
We want to show that 
$$A:=\{n\in\N_0:\, (S_ng,B^nS_ng)\in \mathcal{U}\times \mathcal{V}\}\in\mathcal{F}_{\mathcal{B}}.$$
Let $n_0\in \N$ be large enough so that $\frac{\varepsilon}{|g(a)| |F|}-\delta_{a,n_0}>0$, for every $a\in F$  (where $|F|$ stands for the cardinal of the finite set $F$). Set
$$N=\max\Bigg\{1+\underset{v\in F}{\max}\,\underset{0\leq k<n_0}{\max}\Big(\underset{u\in\Chi^k(v)}{\sum}\dfrac{1}{|\mu_u|^{p^\ast}}\Big)^{1/p^\ast},  \dfrac{1}{ \underset{a\in F}{\min}\big\{\frac{\varepsilon}{|g(a)| |F|}-\delta_{a,n_0}\big\}} \Bigg\}.$$
We will show that $I(F,N)\subset A$.
Note that $B^nS_ne_v=B^ng_{v,n}=e_v$, for every integer $n\geq0$. Indeed, by \eqref{rooted:eqq1}, we have
\begin{align*}
    (B^ng_{v,n})(a)&=\sum_{b\in \Chi^n(a)}g_{v,n}(b)\\
    &=\begin{cases}
        1 & \text{ if } a=v \\
        0 & \text{ if } a\neq v
    \end{cases}\\
    &=e_v(a).
\end{align*}
Thus $B^nS_ng=g\in \mathcal{V}$, for every integer $n\geq0$. Let $n\in I(F,N)$, so as above and by the definition of $N$, we have $n\geq n_0$ and
$$\Big(\underset{u \in \Chi^{n}(a)}{\sum}\,\frac{1}{|\mu_u|^{p^\ast}}\Big)^{-1/p^{\ast}}<\dfrac{1}{N}, \, \forall a\in F,$$
thus by using \eqref{rooted:eqq1}, we obtain:
$$\Big(\sum_{u\in \Chi^{n}(a)}|g_{a,n}(u)\,\mu_u|^p\Big)^{1/p}<\dfrac{1}{N}+\delta_{a,n}, \, \forall a\in F,$$
by using again the definition  of $N$, we deduce 
$$\Big(\sum_{u\in \Chi^{n}(a)}|g_{a,n}(u)\,\mu_u|^p\Big)^{1/p}<\dfrac{\varepsilon}{|g(a)| |F|}, \, \forall a\in F,$$
therefore 
\begin{align*}
    \|S_ng\|_{p,\mu} \leq \sum_{a\in F} |g(a)| \|g_{a,n}\|_{p,\mu}
    =\sum_{a\in F} |g(a)| \Big( \sum_{b\in\Chi^n(a)}|g_{a,n}(b)\mu_b|^p\Big)^{1/p}
    <\varepsilon,
\end{align*}
hence $S_ng\in \mathcal{U}$. We have shown that $n\in A$ and so $I(F,N)\subset A$. Consequently, $A\in\mathcal{F}_\mathcal{B}$. This provides the second condition of the $\mathcal{F}_\mathcal{B}$-transitivity Criterion. Hence $B$ is $\mathcal{F}$-transitive and statement $(2)$ holds.

\smallskip

$(5)\Rightarrow(4)$. Let us assume that $(5)$ holds.
 Let $F$ be a finite subset of $V$  and let $N\in\N$. We want to show that
  $$I(F,N):= \bigcap_{v\in F}\Big\{n\in\N_0:\, \Big(\underset{u\in\Chi^n(v)}{\sum}\dfrac{1}{|\mu_u|^{p^\ast}}\Big)^{1/p^\ast}>N\Big\}\in\mathcal{F}.$$
 Set $g=\sum_{v\in F} e_v\in\ell^p(V,\mu)$. Let $U=\{f\in\ell^p(V,\mu):\, \|f-g\|_{p,\mu}<1/(2N)\}$ and
 $W=\{f\in\ell^p(V,\mu):\, |\langle f- g, e_{v}\rangle|<1/(2N),\, \forall v\in F\}$.
 By the hypothesis, we have $N(U,W)\in \mathcal{F}$ an so $N(U,W)\cap [n,+\infty[\in \mathcal{F}$ for all $n\in\N$. Let  $n_0\in \N$ big enough so that $\Chi^{n}(F)\cap F=\emptyset$, for all $n\geq n_0$. We will show that 
 $$N(U,W)\cap [n_0,+\infty[\subset I(F,N).$$
 Let $n\in N(U,W)\cap [n_0,+\infty[$. Thus $\Chi^{n}(F)\cap F=\emptyset$  and there exists $f\in \ell^p(V,\mu)$ such that  
 \begin{equation}
 \|f- g\|_{p,\mu}<\dfrac{1}{2N} \quad\text{and}\quad|\langle B^{n}f- g, e_{v}\rangle|<\dfrac{1}{2N},\quad \forall v\in F.
     \label{eq54}
 \end{equation}
Fix now $v\in F$. Then the right inequality in \eqref{eq54} gives
\begin{equation*}\label{eq31}
0<\dfrac{1}{2}\leq 1- \dfrac{1}{2N} < \sum_{u\in\Chi^{n}(v)}|f(u)|.
\end{equation*}
 Since $\Chi^{n}(v)\cap F=\emptyset$, by using the left inequality in \eqref{eq54}, we obtain  
 \begin{equation*}\label{eq32}
 \sum_{u\in\Chi^{n}(v)}|f(u)\mu_u|^p<\dfrac{1}{(2N)^p}.
 \end{equation*}
 Using the reverse H\"{o}lder's  inequality (for $p>1$) and the last two inequalities, we can conclude that:
 \begin{align*}
 \dfrac{1}{2N}&>  \Big(\sum_{u\in\Chi^{n}(v)}|f(u)\mu_u|^p\Big)^{1/p}\\
  &\geq \Big(\sum_{u\in\Chi^{n}(v)}|f(u)|\Big) \Big(\sum_{u\in\Chi^{n}(v)}|\mu_u|^{-p^{\ast}}\Big)^{-1/p^{\ast}}\\
  &>\dfrac{1}{2} \Big(\sum_{u\in\Chi^{n}(v)}|\mu_u|^{-p^{\ast}}\Big)^{-1/p^{\ast}},
 \end{align*}
 thus
 $$\Big(\sum_{u\in\Chi^{n}(v)}|\mu_u|^{-p^{\ast}}\Big)^{1/p^{\ast}}>N.$$
This holds for all $v\in F$. Therefore, $n\in I(F,N)$, hence $N(U,W)\cap [n_0,+\infty[\subset I(F,N)$. Consequently, $I(F,N)\in \mathcal{F}$.

\end{proof}

The following example shows that the equivalence between conditions $(4)$ and $(6)$ in Theorem \ref{FTransitive_rooted case} does not hold for Frustenberg families which are not filters.
\begin{example}\label{back_not_hc}
 Let $(V,E)$ be the following rooted directed tree.
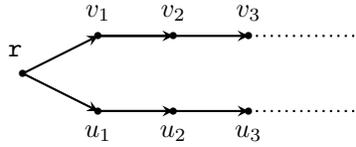
\begin{center}
\begin{tikzpicture}

\draw  (4.9,0.3) node {$\ro$};
\draw  (6,0.8) node {$v_1$};
\draw  (7,0.8) node {$v_2$};
\draw  (8,0.8) node {$v_3$};

\draw  (6,-0.8) node {$u_1$};
\draw  (7,-0.8) node {$u_2$};
\draw  (8,-0.8) node {$u_3$};

\draw (5,0) node {{\tiny$\bullet$}} ;
\draw (6,0.5) node {{\tiny$\bullet$}} ;


\draw (7,0.5) node {{\tiny$\bullet$}} ;
\draw (8,0.5) node {{\tiny$\bullet$}} ;

\draw (6,-0.5) node {{\tiny$\bullet$}} ;
\draw (7,-0.5) node {{\tiny$\bullet$}} ;
\draw (8,-0.5) node {{\tiny$\bullet$}} ;

\tikzstyle{suite}=[->,>=stealth,thick,rounded corners=4pt]


\draw[suite] (5,0) -- (6,0.5) edge (7,0.5) edge (8,0.5) ;
\draw[suite] (5,0) -- (6,-0.5) edge (7,-0.5) edge (8,-0.5);
\draw [dotted, thick] (8,0.5) -- (9.5,0.5);
\draw [dotted, thick] (8,-0.5) -- (9.5,-0.5);
\end{tikzpicture}
\end{center} 
Let $(m_{k})_{k\geq 1}$ be a strictly increasing sequence of non-negative  integers and $\mu=(\mu_v)_{v\in V}$ be the weight defined by $\mu_{\ro}=1$ and, for every $k\geq 1$, $\mu_{u_k}=1/\mu_{v_k}$, where
$$(\mu_{v_k})_{k\geq1}=\Big(\underset{2m_1}{\underbrace{\dfrac{1}{2},\ldots, \dfrac{1}{2^{m_1}},\dfrac{2}{2^{m_1}},\ldots,\dfrac{1}{2},1}},\underset{2m_2}{\underbrace{2,\ldots, 2^{m_2},\dfrac{2^{m_2}}{2},\ldots,2,1}},\ldots\Big).$$
Let $B$ be the backward shift on $\ell^p(V,\mu)$, $1<p<+\infty$. By Proposition \ref{Bounded_Back},  $B$ is bounded and $\|B\|=(2^{p^\ast}+1/2^{p^\ast})^{1/p^\ast}$. Let $\mathcal{J}$ be the Furstenberg family of infinite subsets of $\N_0$. Note that $\mathcal{J}$ is not a filter. By the definition of the weight,  for any $N\in\N$ and $k\in \N$, we have
$$I(\ro,N):=\Big\{n\in\N_0:\, \dfrac{1}{|\mu_{u_n}|^{p^\ast}}+\dfrac{1}{|\mu_{v_n}|^{p^\ast}}>N^{p^\ast}\Big\}\in\mathcal{J},$$
$$I(v_k,N):=\Big\{n\in\N_0:\, \dfrac{1}{|\mu_{v_{k+n}}|}>N\Big\}\in\mathcal{J} \quad \text{ and } \quad I(u_k,N):=\Big\{n\in\N_0:\,\dfrac{1}{|\mu_{u_{k+n}}|}>N\Big\}\in\mathcal{J}.$$
However, $I(u_k,N)\cap I(v_k,N)=\emptyset$, so $I(u_k,N)\cap I(v_k,N)\notin \mathcal{J}.$    
\end{example}

 A bounded linear operator $T$ defined on a Banach space $X$ is said to be weakly hypercyclic if there exists a vector $x \in X$ such that its orbit $\mathrm{Orb}(x,T)$ is dense in $X$ with respect to the weak topology. 
 
 When the Furstenberg family $\mathcal{F}$ in the previous theorem is chosen as the collection of syndetic subsets of $\N_0$, we can derive a characterization of topologically ergodic backward shift operators on rooted directed trees. Furthermore, when $\mathcal{F}$ represents the family of infinite subsets, $\mathcal{F}$-transitivity coincides with hypercyclicity. This allows us to deduce the following corollary.

\begin{corollary}\label{hc_rooted}
    Let $(V,E)$ be a rooted directed tree and let $\mu=(\mu_v)_{v\in V}$ be a weight on $V$.  Let $X=\ell^p(V,\mu)$, $1\leq p <+\infty$, or $X=c_0(V,\mu)$ and suppose  that the  backward shift $B$ is a bounded operator on $X$. Then the following assertions are equivalent: 
 \begin{enumerate}[label={$(\arabic*)$}]
 \item $B$ is hypercyclic.
 \item $B$ is weakly hypercyclic. 
 \item There exists a bounded subset $C$ of $X\setminus\{0\}$ such that $\Orb(C,B)$ is weakly dense in $X$.
 \item For every $g\in X$, for every open neighborhood $U$ of $g$, for every weakly open neighborhood $W$ of $g$, there exists an integer $n\geq1$ such that
 $$B^n(U)\cap W\neq\emptyset.$$
 \end{enumerate}
\end{corollary}

 Note that the implication $(2)\Rightarrow (4)$ in the previous corollary holds for any operator (see \cite[Page 42]{CSa}). Furthermore, it is worth mentioning that if there exists an integer $n$ satisfying condition $(4)$, then the return set from $U$ to $W$ is infinite.

\begin{remark}
  Let $X$ be a Banach space and $T\in\mathcal{L}(X)$ be a linear operator. If for any nonempty weakly open subset $U$ of $X$, there exists an integer $n\geq 1$ such that 
    $$T^n(U)\cap U\neq \emptyset,$$ 
    then there are infinitely many integers that satisfy this property.
To see this, note that the range of $T$ is weakly dense in $X$ since it meets any nonempty weakly open subset of $X$. By Mazur's theorem, the weak and norm closures of the range of $T$ coincide, which implies that $T$ has a dense range.

Now, let $U$ be a nonempty weakly open subset of $X$, $x\in U$, and let $n\geq 1$ be an integer such that $T^n x\in U$. Then, the set $W=U\cap T^{-n}(U)$ is a nonempty weakly open subset of $X$. By the hypothesis, there exists an integer $m\geq 1$ such that $T^m(W)\cap W\neq \emptyset$, and so we have
 $$T^{n+m}(U)\cap U\neq\emptyset.$$
Repeating this argument, we conclude that there exist infinitely many integers $k\geq 1$ such that $T^k(U)\cap U\neq\emptyset$.
\end{remark}

In the case of rooted trees, by the previous corollary, recurrence and hypercyclicity coincide for backward shifts. This equivalence, in fact, holds for any bounded linear operator $T \in \mathcal{L}(X)$ whose generalized kernel (or, more generally, the set of vectors $x \in X$ such that $T^nx\longrightarrow0$ as $n\longrightarrow+\infty$) is dense in the underlying space (see \cite[Theorem 2.12]{BGLP}). Now, we will establish a similar result for $\mathcal{F}$-transitivity.

\begin{proposition}
 Let $X$ be a Banach space, $T\in\mathcal{L}(X)$ and let $\mathcal{F}$ be a Furstenberg family on $\N_0$ such that $A\cap [n, +\infty[\in\mathcal{F}$ whenever $A\in\mathcal{F}$ and $n\in \N$. Suppose that the set
 $$X_0:=\{x\in X:\, \underset{n\to+\infty}{\lim} T^nx=0 \}$$
is dense in $X$. Then the following statements are equivalent:
 \begin{enumerate}[label={$(\arabic*)$}]
     \item $T$ is $\mathcal{F}$-transitive.
     \item $T$ is topologically $\mathcal{F}$-recurrent. 
 \end{enumerate}
\end{proposition}
\begin{proof}
It is clear that $(1)\Rightarrow (2)$. Assume that $(2)$ holds. Let $U,V$ be nonempty open subsets of $X$.  There exist a nonempty open subset $V'$ of $X$ and
a $0$-neighbourhood $W$ such that  $W+V' \subset V$. Assume that $U\neq V'$, otherwise there is nothing to prove. Let $y\in(U-V')\cap X_0$. There exists then a non-negative integer $N$ such that $T^ny\in W$, for all $n\geq N$. By the hypothesis, we have $N(V',V')\in\mathcal{F}$ and so $N(V',V')\cap [N,+\infty[\in\mathcal{F}$.  Now, we will show that 
$$N(V',V')\cap [N,+\infty[ \subset N(U,V).$$
Let $n\in N(V',V')\cap [N,+\infty[$. Thus $T^nV'\cap V'\neq \emptyset$ and $T^ny\in W$. There exists then $z\in V'$ such that $T^nz\in V'$. Hence $x=y+z\in U$ and 
$$T^nx=T^ny+T^nz\in W+V'\subset V,$$
therefore, $n\in N(U,V)$. Consequently, $N(V',V')\cap [N,+\infty[\subset N(U,V)$ and so $N(U,V)\in\mathcal{F}$. This finishes the proof.
\end{proof}

\section{$\mathcal{F}$-transitivity - unrooted case} \label{section5}

We will now present our second main result, which characterizes $\mathcal{F}$-transitivity for backward shifts on unrooted directed trees and establishes its equivalence with topological $\mathcal{F}$-recurrence.

\begin{theorem}\label{FTransitive_unrooted case}
 Let $(V,E)$ be an unrooted directed tree, let $\mu=(\mu_v)_{v\in V}$ be a weight on $V$, and let $\mathcal{F}$ be a Furstenberg family on $\N_0$. Let $X=\ell^p(V,\mu)$, $1\leq p <+\infty$, or $X=c_0(V,\mu)$ and suppose  that the  backward shift $B$ is a bounded operator on $X$. The following assertions are equivalent: 
 \begin{enumerate}[label={$(\arabic*)$}]
  \item $B$ is $\widetilde{\mathcal{F}}$-transitive.
 \item $B$ is $\mathcal{F}$-transitive.
 \item For every $N\in\N$ and every finite subset $F\subset V$, we have $I(F,N)\cap J(F,N)\in \mathcal{F}$, where
 $$I(F,N):=\begin{cases}
   \underset{v\in F}{ \bigcap}\Big\{n\in\N_0:\, \underset{u\in\Chi^n(v)}{\sup}\dfrac{1}{|\mu_u|}>N\Big\} & \text{if } X=\ell^1(V,\mu);\\
   \underset{v\in F}{ \bigcap}\Big\{n\in\N_0:\, \Big(\underset{u\in\Chi^n(v)}{\sum}\dfrac{1}{|\mu_u|^{p^\ast}}\Big)^{1/p^\ast}>N\Big\} & \text{if }X=\ell^p(V,\mu),\, 1<p<+\infty;\\
   \underset{v\in F}{ \bigcap}\Big\{n\in\N_0:\, \underset{u\in\Chi^n(v)}{\sum}\dfrac{1}{|\mu_u|}>N\Big\} & \text{if }X=c_0(V,\mu);
\end{cases}$$
and
$$J(F,N):=\begin{cases}
  \underset{v\in F}{ \bigcap}\Big\{n\in\N_0:\, \min\big(|\mu_{\p^{n}(v)}|,\underset{u\in \Chi^{n}(\p^{n}(v))}{\inf}\,|\mu_u|\big)<\frac{1}{N}\Big\} & \text{if } X=\ell^1(V,\mu);\\
   \underset{v\in F}{ \bigcap}\Big\{n\in\N_0:\, \dfrac{1}{|\mu_{\p^{n}(v)}|^{p^\ast}}+\underset{u \in \Chi^{n}(\p^{n}(v))}{\sum}\dfrac{1}{|\mu_u|^{p^{\ast}}}>N^{p^\ast}\Big\} & \text{if }X=\ell^p(V,\mu),\, 1<p<+\infty;\\
   \underset{v\in F}{ \bigcap}\Big\{n\in\N_0:\, \dfrac{1}{|\mu_{\p^{n}(v)}|}+\underset{u \in \Chi^{n}(\p^{n}(v))}{\sum}\dfrac{1}{|\mu_u|}>N\Big\} & \text{if }X=c_0(V,\mu).
\end{cases}$$
 \end{enumerate}
If $\mathcal{F}$ satisfies that $A\cap [n,+\infty[ \in\mathcal{F}$ whenever $A\in\mathcal{F}$ and $n\in\N$, then the above conditions are equivalent to
\begin{enumerate}
    \item[$(4)$] $B$ is topologically $\mathcal{F}$-recurrent.
\end{enumerate} 
\end{theorem}
\begin{proof}
    We will only prove the equivalences in the case where $X = \ell^p(V,\mu)$ with $1 < p < +\infty$. A similar argument can be made to deduce the cases where $X = \ell^1(V,\mu)$ and $X = c_0(V,\mu)$. As in the proof of Theorem \ref{FTransitive_rooted case} and by using \cite[Theorem 5.2]{GrPa}, we obtain that $(1)\Longleftrightarrow (2)$. Note that, it is clear that 
$(2)\Rightarrow (4)$.

     Let us  show that $(4)\Rightarrow (3)$.  Let $F$ be a finite subset of $V$  and let $N\in\N$.  Set $g=\sum_{v\in F} e_v\in\ell^p(V,\mu)$. Let $U=\{f\in\ell^p(V,\mu):\, \|f-g\|_{p,\mu}<\frac{1}{2N}\}$. Let  $n_0\in \N$ big enough so that $\Chi^{n}(F)\cap F=\emptyset$ and $\p^{n}(F)\cap F=\emptyset$, for all $n\geq n_0$. By the hypothesis, $N(U,U)\in\mathcal{F}$ and so $N(U,U)\cap [n_0,+\infty[\in\mathcal{F}$. Since $\mathcal{F}$ is a Furstenberg family, it is enough to prove that 
     $$N(U,U)\cap [n_0,+\infty[\subset I(F,N)\cap J(F,N).$$
     We can assume that $N$ is sufficiently large such that  $1<|\mu_v|N$, $\forall v \in F$. Let $n\in N(U,U)\cap [n_0,+\infty[$. There exists then a function $f$  in $\ell^p(V,\mu)$ such that 
\begin{equation}
\|f- g\|_{p,\mu}<\dfrac{1}{2N} \quad\text{and}\quad\|B^{n}f- g\|_{p,\mu}<\dfrac{1}{2N}.
\label{eqq45}
\end{equation}
Fix now $v\in F$.  By the right above inequality,  we have 
\begin{equation}\label{eqq43}
0<\dfrac{1}{2}<1-\dfrac{1}{2N|\mu_v|}< \sum_{u\in\Chi^{n}(v)}|f(u)|.
\end{equation}
Since $\Chi^{n}(v)\cap F=\emptyset$ and $F=\supp(g)$, by the left inequality in \eqref{eqq45}, we obtain
\begin{equation}\label{eqq44}
 \sum_{u\in \Chi^{n}(v) } |f(u)\,\mu_u|^p<\dfrac{1}{(2N)^p}.   
\end{equation}
 By the reverse H\"{o}lder's inequality ($p>1$) and \eqref{eqq43}, we get
 \begin{align*}
 \sum_{u\in\Chi^{n}(v)}|f(u)\mu_u|^p &\geqslant \Big(\sum_{u\in\Chi^{n}(v)}|f(u)|\Big)^{p} \Big(\sum_{u\in\Chi^{n}(v)}|\mu_u|^{\frac{p}{1-p}}\Big)^{1-p}\\
 &\geqslant \dfrac{1}{2^p} \Big(\sum_{u\in\Chi^{n}(v)}|\mu_u|^{\frac{p}{1-p}}\Big)^{1-p},
 \end{align*}
 combining this with \eqref{eqq44}, we get
 $$\Big(\sum_{u\in\Chi^{n}(v)}|\mu_u|^{-p^{\ast}}\Big)^{1/p^{\ast}}>N,$$
this holds for any fixed $v\in F$, hence $n\in I(F,N)$. Let us also show that $n\in J(F,N)$. By contradiction, suppose that  there is a $v\in F$ for which  it holds
\begin{equation}
  \dfrac{1}{|\mu_{\p^{n}(v)}|^{p^\ast}}+\underset{u \in \Chi^{n}(\p^{n}(v))}{\sum}\dfrac{1}{|\mu_u|^{p^{\ast}}}\leq  N^{p^\ast}.  
  \label{eqq1}
\end{equation}
 Thus $\frac{1}{N}\leqslant  |\mu_{\p^{n}(v)}|$.  Set 
  $$h=(f-g)\chi_{\Chi^{n}(\p^{n}(v))}.$$
  One has
  \begin{equation}
      \|h\|_{p,\mu}\leqslant \|f-g\|_{p,\mu}<\dfrac{1}{2N}.
      \label{eqq46}
  \end{equation}
  Note now that,  we have
  $$(B^{n}f)(\p^{n}(v))=d_n+\sum_{u\in\Chi^{n}(\p^{n}(v))}h(u),$$
  where $d_n:=\big|\Chi^{n}(\p^{n}(v))\cap F\big|\geq1$, where $|\cdot|$ stands for the cardinal. Since $\p^{n}(v)\in V\setminus F$, by using the right inequality in \eqref{eqq45}, we get
 \begin{align*}
     \dfrac{1}{2N} &>|(B^{n}f)(\p^{n}(v))\,\mu_{\p^{n}(v)}|\\
     &\geq \Big(d_n-\big|\sum_{u\in\Chi^{n}(\p^{n}(v))}h(u)\big|\Big)|\mu_{\p^{n}(v)}|\\
     & \geq \dfrac{1}{N}\Big(1-|\sum_{u\in\Chi^{n_k}(\p^{n_k}(v))}h_k(u)|\Big),
 \end{align*}
 hence
 $$\dfrac{1}{2}<\sum_{u\in\Chi^{n}(\p^{n}(v))}|h(u)|.$$
 By the reverse H\"{o}lder's inequality and \eqref{eqq46}, we have 
 \begin{align*}
\dfrac{1}{2N}&> \Big(\sum_{u\in\Chi^{n}(\p^{n}(v))}|h(u)|\Big) \Big(\sum_{u\in\Chi^{n}(\p^{n}(v))}|\mu_u|^{-p^\ast}\Big)^{-1/p^{\ast}}\\
&> \dfrac{1}{2} \Big(\sum_{u\in\Chi^{n}(\p^{n}(v))}|\mu_u|^{-p^\ast}\Big)^{-1/p^{\ast}},
 \end{align*}
then
$$\underset{u \in \Chi^{n}(\p^{n}(v))}{\sum}\dfrac{1}{|\mu_u|^{p^{\ast}}}>N^{p^\ast}.$$
This contradicts \eqref{eqq1}. Therefore, $n\in J(F,N)$. Consequently, $N(U,U)\cap [n_0,+\infty[\subset I(F,N)\cap J(F,N)$, and so $I(F,N)\cap J(F,N)\in\mathcal{F}$. 

\smallskip

Note that, a minor adjustment to the proof of the implication $(4)\Rightarrow (3)$ allows us to deduce $(2)\Rightarrow (3)$. Let us show now that $(3)\Rightarrow (2)$.   Let $\mathcal{B}$ be the filter base consisting of all subsets $I(F,N)\cap J(F,N)$ of $\N_0$, where  $N>0$ and  $F\subset V$ is finite. Let $\mathcal{F}_\mathcal{B}$ be the filter generated by $\mathcal{B}$, that is,
   $$\mathcal{F}_{\mathcal{B}}=\{A\subset \N_0: B\subset  A \text{ for some } B\in \mathcal{B}\}.$$ 
   We will show that $B$ satisfies the $\mathcal{F}_{\mathcal{B}}$-transitivity Criterion, which implies that $B$ is $\mathcal{F}$-transitive since $\mathcal{F}_{\mathcal{B}}\subset \mathcal{F}$.
   Set $\mathcal{D}=\mathrm{span}\{e_v:\, v\in V\}$, which is dense in $\ell^p(V,\mu)$. For every $v\in V$ and $n\in \N$, by the continuity of $B^{n}$, we have
        $$M_{v,n}:=\dfrac{1}{|\mu_{\p^{n}(v)}|^{p^\ast}}+\underset{u \in \Chi^{n}(\p^{n}(v))}{\sum}\dfrac{1}{|\mu_u|^{p^{\ast}}}<+\infty,$$
        then either 
        \begin{equation}\label{eqq40}
           \dfrac{1}{|\mu_{\p^{n}(v)}|^{p^\ast}}\geq \dfrac{M_{v,n}}{2}, 
        \end{equation}
        or
        \begin{equation}\label{eqq41}
          \underset{u \in \Chi^{n}(\p^{n}(v))}{\sum}\dfrac{1}{|\mu_u|^{p^{\ast}}}>\dfrac{M_{v,n}}{2}.  
        \end{equation}
        In the case where \eqref{eqq40} holds, we set
        $$I_{n}e_v=e_v,$$
        and then
        \begin{equation}
            \|I_{n}e_v-e_v\|_{p,\mu}=0 \quad \text{and} \quad\|B^{n}I_{n} e_v\|_{p,\mu}=\|e_{\p^{n}(v)}\|_{p,\mu}=| \mu_{\p^{n}(v)}|\leq \Big(\dfrac{2}{M_{v,n}}\Big)^{1/p^\ast}.
            \label{eqq3}
        \end{equation}
        In the case where \eqref{eqq41} holds, by Lemma \ref{Lemma_Grosse}, the there exists $h_{v,n}\in\K^V$ non-negative, of support in $\Chi^{n}(\p^{n}(v))$ such that 
\begin{equation}\label{eq42}
    \sum_{u\in \Chi^{n}(\p^{n}(v))}h_{v,n}(u)=1 \quad \text{and} \quad \Big(\sum_{u\in\Chi^{n}(\p^{n}(v))}|h_{v,n}(u)\,\mu_u|^p\Big)^{1/p}\leq \Big(\dfrac{2}{M_{v,n}}\Big)^{1/p^\ast}.
\end{equation}
Hence, $h_{v,n}\in \ell^p(V,\mu)$, and by setting
    $$I_{n}e_v=e_v-h_{v,n},$$
we obtain
\begin{equation}
    \|I_{n}e_v-e_v\|_{p,\mu}\leq \Big(\dfrac{2}{M_{v,n}}\Big)^{1/p^\ast}  \quad \text{and} \quad B^{n}I_{n} e_v=0.
    \label{eqq4}
\end{equation}
In both cases, we extend linearly on $\mathcal{D}$ the maps $I_{n}$.  Let  $g=\sum_{a\in F}g(a) e_a\in \mathcal{D}$, where $F\subset V$ is its support.
We will show that $\mathcal{F}_{\mathcal{B}}\text{-}\underset{n}{\lim}(I_ng,B^nI_ng)=(g,0)$. Let $\varepsilon>0$, $\mathcal{U}:=\{f\in\ell^p(V,\mu):\, \|f-g\|_{p,\mu}<\varepsilon\}$ and $\mathcal{V}:=\{f\in\ell^p(V,\mu):\, \|f\|_{p,\mu}<\varepsilon\}$. Let us check that
\begin{equation}
   J(F,N)\subset\{n\in\N_0:\, (I_ng,B^nI_ng)\in \mathcal{U}\times \mathcal{V}\}, 
   \label{eqq5}
\end{equation}
where
$$N=\max\Big\{\frac{|g(a)|\,|F|2^{1/p^\ast}}{\varepsilon}:\, a\in F\Big\}.$$ 
Let $n\in J(F,N)$. We have then $M_{a,n}>N^{p^\ast}$, for all $a\in F$. By the definition of $N$, \eqref{eqq3} and \eqref{eqq4}, we obtain:
\begin{align*}
\max \Big\{\|I_{n}g-g\|,\|B^{n}I_{n}g\|_{p,\mu}  \Big\} &\leq \sum_{a\in F}|g(a)| \Big(\dfrac{2}{M_{a,n}}\Big)^{1/p^\ast}\\     
&<  \sum_{a\in F}|g(a)| \dfrac{2^{1/p^\ast}}{N}\\
&< \varepsilon,
\end{align*}
hence $I_ng\in \mathcal{U}$ and $B^nI_ng\in \mathcal{V}$. Therefore, \eqref{eqq5} holds and so 
$$\{n\in\N_0:\, (I_ng,B^nI_ng)\in \mathcal{U}\times \mathcal{V}\}\in\mathcal{F}_\mathcal{B}.$$
Consequently, the first statement of the $\mathcal{F}_\mathcal{B}$-transitivity Criterion holds. As for the second statement, we can define the maps $S_n$ as in the proof of Theorem \ref{FTransitive_rooted case}. By employing the same arguments, we can deduce that
$$\{n\in\N_0:\, (B^nS_ng,S_ng)\in \mathcal{U}\times \mathcal{V}\}\in\mathcal{F}_{\mathcal{B}}.$$
This shows that the second statement of the $\mathcal{F}_\mathcal{B}$-transitivity Criterion holds. Hence $B$ is $\mathcal{F}$-transitive and statement $(2)$ holds.

\end{proof}

Once again, if the Furstenberg family $\mathcal{F}$ in the previous theorem represents the collection of syndetic subsets, we can derive a  characterization of topologically ergodic backward shift operators on unrooted directed trees. Additionally, in the case where $\mathcal{F}$ represents the family of infinite subsets, we deduce the following corollary.

\begin{corollary}
    Let $(V, E)$ be an unrooted directed tree and $\mu=(\mu_v)_{v\in V}$ a weight on $V$.  Let $X=\ell^p(V,\mu)$, $1\leq p <+\infty$, or $X=c_0(V,\mu)$ and suppose  that the  backward shift $B$ is a bounded operator on $X$. Then $B$ is hypercyclic if and only if $B$ is recurrent.
\end{corollary}

In the upcoming section, we will generalize this corollary by using the concept of $\Gamma$-supercyclicity.

\section{$\Gamma$-supercyclicity} \label{section6}
Supercyclicity is a weaker notion than hypercyclicity, which requires the density of a projective orbit instead of an orbit, see \cite{BaMa}. We can study these both notions by considering the notion of $\Gamma$-supercyclicity, see  \cite{CEM}. Let $X$ be a Banach space and $\Gamma$ be a subset of $\C$. An operator $T\in\mathcal{L}(X)$ is called \textit{$\Gamma$-supercyclic} if there exists some vector $x\in X$ such that the set
\[\Orb(\Gamma x,T):=\{\lambda T^nx:\, \lambda\in\Gamma,\, n\in\N_0\}\]
is dense in $X$. In particular, $\C$-supercyclicity is called simply supercyclicity, and $\{1\}$-supercyclicity coicides with hypercyclicity. This notion has been studied for bilateral backward shifts in \cite{Ab1}, and more generally for a family of translation operators on weighted $L^p$-spaces on locally compact groups in \cite{AbKu}.

In the same way, we can extend the notion of recurrence in the following way: a vector $x\in X$ is said to be \textit{$\Gamma$-recurrent vector} for $T$ if there exist a strictly increasing sequence $(n_k)_{k\geq 1}$ of positive integers and a sequence $(\lambda_k)_{k\geq1}$ in $\Gamma$ such that
\[\lambda_k T^{n_k}x\underset{k\to+\infty}{\longrightarrow}x.\]
The operator $T$ is called $\Gamma$-recurrent if its set of $\Gamma$-recurrent vectors, denoted by $\Gamma$\text{-}$\mathrm{Rec}(T)$, is dense in $X$.  This notion coincides with the so-called super-recurrent when $\Gamma$ is equal to the complex plane \cite{AmBe2}.

 In the following theorem, we provide a characterization of $\Gamma$-supercyclicity for backward shift operators on unrooted directed trees and give its equivalence with $\Gamma$-recurrence. The proof of this theorem is omitted, as it follows a similar approach to that of Theorem \ref{FTransitive_unrooted case}, specifically when considering the Furstenberg family $\mathcal{F}$ as the collection of infinite subsets of $\mathbb{N}_0$. Furthermore, instead of using the $\mathcal{F}$-transitivity criterion, we employ the $\Gamma$-supercyclicity criterion described below.

\begin{theorem}\label{super_criterion} 
 Let $(V, E)$ be an unrooted directed tree and $\mu=(\mu_v)_{v\in V}$ a weight on $V$.  Let $X=\ell^p(V,\mu)$, $1\leq p <+\infty$, or $X=c_0(V,\mu)$ and suppose  that the  backward shift $B$ is a bounded operator on $X$. Then the following conditions are equivalent:
 \begin{enumerate}[label={$(\arabic*)$}]
 \item $B$ is $\Gamma$-supercyclic.
 \item $B$ is $\Gamma$-recurrent.
 \item  There are an increasing sequence of positive integers $(n_k)_{k\geq 1}$  and a sequence $(\lambda_k)_{\geq 1}$ in $\Gamma\setminus\{0\}$ such that, for every $v\in V$, we have
 $$\begin{cases}
\underset{u \in \Chi^{n_k}(v)}{\sum}\dfrac{|\lambda_k|^{p^{\ast}}}{|\mu_u|^{p^{\ast}}}\underset{k\to+\infty}{\longrightarrow}+\infty & \text{ if } X=\ell^p(V,\mu),\, 1< p<+\infty;\\
 &\\
\underset{u\in \Chi^{n_k}(v)}{\inf}\,\dfrac{|\mu_u|}{|\lambda_k|}\underset{k\to+\infty}{\longrightarrow}0 &\text{ if } X=\ell^1(V,\mu);\\
&\\
\underset{u \in \Chi^{n_k}(v)}{\sum}\dfrac{|\lambda_k|}{|\mu_u|}\underset{k\to+\infty}{\longrightarrow}+\infty&\text{ if } X=c_0(V,\mu),
\end{cases}$$
 and
   $$\begin{cases}
\dfrac{1}{|\lambda_k\,\mu_{\p^{n_k}(v)}|^{p^\ast}}+\underset{u \in \Chi^{n_k}(\p^{n_k}(v))}{\sum}\dfrac{1}{|\mu_u|^{p^{\ast}}}\underset{k\to+\infty}{\longrightarrow}+\infty & \text{ if } X=\ell^p(V,\mu),\, 1< p<+\infty;\\
 &\\
\min\Big(|\lambda_k\,\mu_{\p^{n_k}(v)}|,\underset{u\in \Chi^{n_k}(\p^{n_k}(v))}{\inf}\,|\mu_u|\Big)\underset{k\to+\infty}{\longrightarrow}0 &\text{ if } X=\ell^1(V,\mu);\\
&\\
\dfrac{1}{|\lambda_k\, \mu_{\p^{n_k}(v)}|}+\underset{u \in \Chi^{n_k}(\p^{n_k}(v))}{\sum}\dfrac{1}{|\mu_u|}\underset{k\to+\infty}{\longrightarrow}+\infty&\text{ if } X=c_0(V,\mu)
\end{cases}.$$
 \end{enumerate}

\end{theorem}

The proof of the $\Gamma$-supercyclicity criterion presented below is a straightforward adaptation of the proof of the supercyclicity criterion given in \cite[Theorem 1.14]{BaMa}.
\begin{theorem}[$\Gamma$-supercyclicity  Criterion]
 Let $X$ be an infinite-dimensional Banach space, $T\in \mathcal{L}(X)$ and let $\Gamma\subset\C$ be such that $\Gamma\setminus\{0\}$ is non-empty. Assume that there exist an increasing sequence of positive integers $(n_k)_{k\geq1}$, a sequence $(\lambda_k)_{k\geq1}$ in $\Gamma\setminus\{0\}$, two dense subsets $X_0$ and $Y_0$ in $X$, and two families of applications $I_{n_k}:X_0\to X$ and $S_{n_k}:Y_0\to X$ such that, for any $x\in X_0$ and any $y\in Y_0$, the following conditions hold:
 \begin{enumerate}[label=$(\alph*)$]
     \item $I_{n_k}x\underset{k\to+\infty}{\longrightarrow} x$;
     \item $\lambda_k T^{n_k} I_{n_k}x\underset{k\to+\infty}{\longrightarrow} 0$;
     \item $\dfrac{1}{\lambda_k} S_{n_k} y\underset{k\to+\infty}{\longrightarrow} 0$;
     \item $T^{n_k}S_{n_k} y\underset{k\to+\infty}{\longrightarrow} y$.
 \end{enumerate}
 Then $T\oplus T$ is $\Gamma$-supercyclic.
 \label{Criterion}
\end{theorem}

In the case of rooted trees, we can deduce from Corollary \ref{hc_rooted} that for any bounded subset $\Gamma$ of $\C\setminus\{0\}$, a backward shift operator $B$ is hypercyclic if and only if it is $\Gamma$-supercyclic.  In addition, in the case where $\Gamma$ is  an unbounded subset of $\C$, $B$ is always $\Gamma$-supercyclic, because in this case $B$ has a dense generalized kernel and a dense range, see Proposition \ref{gen_ker}  below. 

The generalized kernel of  a bounded linear operator $T$ defined on some Banach space $X$ is the subspace $$\ker^\ast(T):=\bigcup_{n=1}^{+\infty}\ker(T^n).$$
In \cite[Corollary 3.3]{BBP}, it was proven that an operator with a dense generalized kernel is supercyclic if and only if it has a dense range. 
\begin{proposition}\label{gen_ker}
Let $X$ be a separable infinite-dimensional Banach space and $T\in\mathcal{L}(X)$ with a dense generalized kernel. Then the following assertions are equivalent:
\begin{enumerate}[label={$(\arabic*)$}]
    \item For any unbounded subset $\Gamma$ of $\C$, $T$ is $\Gamma$-supercyclic. 
    \item There exists an unbounded subset $\Gamma$ of $\C$ such that $T$ is $\Gamma$-supercyclic. 
    \item $T$ has a dense range.
\end{enumerate}
\end{proposition}
\begin{proof}
It is clear that $(1)\Rightarrow (2)\Rightarrow (3)$. Let us show that $(3)\Rightarrow (1)$. Assume that $T$ has a dense range. Let $\Gamma$ be an unbounded subset of $\C$. We will show that $T\oplus T$ is $\Gamma$-supercyclic. Let $U_1$, $U_2$, $V_1$, $V_2$ be nonempty open subsets of $X$. Since $\ker^\ast(T)$ is dense in $X$, there exist $(x_1,x_2)\in U_1\times U_2$ such that $T^nx_1=T^nx_2=0$ for some $n\in\N$. Since $T$ has a dense range, $T^n$ also is, therefore, there exist $y_1,y_2\in X$ such that $(T^ny_1,T^ny_2)\in V_1\times V_2$. By the unboundedness of $\Gamma$, there exists $\lambda\in\Gamma$ such that
$$\dfrac{1}{\lambda}(y_1,y_2)+(x_1,x_2)\in U_1\times U_2,$$
and
$$\lambda(T\oplus T)^n(\dfrac{1}{\lambda}(y_1,y_2)+(x_1,x_2))=(T^ny_1,T^ny_2)+\lambda(T^nx_1,T^nx_2)\in V_1\times V_2.$$
Hence $T\oplus T$ is $\Gamma$-supercyclic. 
\end{proof}

\section{Zero-one law limit point}
In \cite{CS}, Chan and Seceleanu showed that any weighted unilateral/bilateral backward shift that has an orbit with a nonzero limit point is hypercyclic. However, comparing the characterization of hypercyclicity for backward shift operators on rooted directed trees 
with the following theorem, we deduce that Chan and Seceleanu's result does not hold for  rooted directed trees.

\begin{theorem}\label{recu_tree_shift}
 Let $(V,E)$ be a rooted directed tree and let $\mu=(\mu_v)_{v\in V}$ be a weight on $V$.  Let $X=\ell^p(V,\mu)$, $1\leq p <+\infty$, or $X=c_0(V,\mu)$ and suppose  that the  backward shift $B$ is a bounded operator on $X$. The following conditions are equivalent:
 \begin{enumerate}[label={$(\arabic*)$}]
 \item $B$ has an orbit with a nonzero limit point.
  \item $B$ has an orbit with a nonzero weak limit point.
  \item $B$ has an orbit with  $e_{\ro}$ as a limit point.
 \item $B$ has an orbit with  $e_{\ro}$ as a weak limit point.
 \item  There are an increasing sequence $(n_k)_{k}$ of positive integers and a vertex $v\in V$ such that
 $$\begin{cases}
\underset{u \in \Chi^{n_k}(v)}{\sum}\dfrac{1}{|\mu_u|^{p^{\ast}}}\underset{k\to+\infty}{\longrightarrow}+\infty & \text{ if } X=\ell^p(V,\mu),\, 1< p<+\infty;\\
 &\\
\underset{u\in\Chi^{n_{k}}(v)}{\inf}\,|\mu_u|\underset{k\to+\infty}{\longrightarrow}0 &\text{ if } X=\ell^1(V,\mu); \\
&\\
\underset{u \in \Chi^{n_k}(v)}{\sum}\dfrac{1}{|\mu_u|}\underset{k\to+\infty}{\longrightarrow}+\infty&\text{ if } X=c_0(V,\mu).
\end{cases}$$
  \item  There is an increasing sequence $(n_k)_{k}$ of positive integers  such that
  $$\begin{cases}
\underset{u \in \Chi^{n_k}(\ro)}{\sum}\dfrac{1}{|\mu_u|^{p^{\ast}}}\underset{k\to+\infty}{\longrightarrow}+\infty & \text{ if } X=\ell^p(V,\mu),\, 1< p<+\infty;\\
 &\\
\underset{u\in\Chi^{n_{k}}(\ro)}{\inf}\,|\mu_u|\underset{k\to+\infty}{\longrightarrow}0 &\text{ if } X=\ell^1(V,\mu); \\
&\\
\underset{u \in \Chi^{n_k}(\ro)}{\sum}\dfrac{1}{|\mu_u|}\underset{k\to+\infty}{\longrightarrow}+\infty&\text{ if } X=c_0(V,\mu),\,
\end{cases}.$$ 
\item There exist an increasing sequence $(n_k)_{k\in\N}$ of positive integers and a vertex $v\in V$ such that, for any $l\in\N_0$,
 $$\begin{cases}
\underset{u \in \Chi^{n_k+l}(v)}{\sum}\dfrac{1}{|\mu_u|^{p^{\ast}}}\underset{k\to+\infty}{\longrightarrow}+\infty & \text{ if } X=\ell^p(V,\mu),\, 1< p<+\infty;\\
 &\\
\underset{u\in\Chi^{n_{k}+l}(v)}{\inf}\,|\mu_u|\underset{k\to+\infty}{\longrightarrow}0 &\text{ if } X=\ell^1(V,\mu); \\
&\\
\underset{u \in \Chi^{n_k+l}(v)}{\sum}\dfrac{1}{|\mu_u|}\underset{k\to+\infty}{\longrightarrow}+\infty&\text{ if } X=c_0(V,\mu).
\end{cases}$$
 \end{enumerate}
\end{theorem}
\begin{proof}
It is enough to prove that $(2) \Rightarrow (5) \Rightarrow (6) \Rightarrow (3)$ and $(5)\Rightarrow (7)$. We will prove these implications when $X=\ell^p(V,\mu)$, with $1<p<+\infty$, the same reasoning works for the other cases. 
Let us show that $(2)$ implies $(5)$.  Let $f,g\in \ell^p(V,\mu)$ be non-zero vectors. Let $v_0\in V$ such that $g(v_0)\neq0$. Let $(\delta_k)_{k}$ be a decreasing sequence of positive numbers such that $2\delta_k<|g(v_0)|$. Let $(n_k)_{k}$ be an increasing sequence of non-negative integers such that
$$|\langle B^{n_k}f- g, e_{v_0}\rangle|<\delta_k,$$
that is
$$|(B^{n_k}f)(v_0)-g(v_0)|<\delta_k,$$
thus
\begin{equation}\label{eq4}
0<\dfrac{|g(v_0)|}{2}<|g(v_0)|-\delta_k< \sum_{u\in\Chi^{n_k}(v_0)}|f(u)|.
\end{equation}
 Since  $\|f\|_{p,\mu}<+\infty$, we obtain
 $$\sum_{k\geq0}\sum_{u\in\Chi^{n_k}(v_0)}|f(u)\mu_u|^p<+\infty,$$
 hence
 \begin{equation}\label{eq5}
 \sum_{u\in\Chi^{n_k}(v_0)}|f(u)\mu_u|^p\underset{k\to+\infty}{\longrightarrow}0.
 \end{equation}
 By the reverse H\"{o}lder's inequality ($p>1$), we get
 \begin{align*}
 \sum_{u\in\Chi^{n_k}(v_0)}|f(u)\mu_u|^p &\geqslant \Big(\sum_{u\in\Chi^{n_k}(v_0)}|f(u)|\Big)^{p} \Big(\sum_{u\in\Chi^{n_k}(v_0)}|\mu_u|^{\frac{p}{1-p}}\Big)^{1-p}\\
 &\overset{\eqref{eq4}}{\geqslant} \dfrac{|g(v_0)|^p}{2^p} \Big(\sum_{u\in\Chi^{n_k}(v_0)}|\mu_u|^{\frac{p}{1-p}}\Big)^{1-p}
 \end{align*}
hence
 $$\Big(\sum_{u\in\Chi^{n_k}(v_0)}|\mu_u|^{-p^{\ast}}\Big)^{-1/p^{\ast}}\leqslant \dfrac{2}{|g(v_0)|} \Big(\sum_{u\in\Chi^{n_k}(v_0)}|f(u)\mu_u|^p\Big)^{1/p},$$
 combining this with \eqref{eq5}, we get
  $$\sum_{u \in \Chi^{n_k}(v_0)}\dfrac{1}{|\mu_u|^{p^{\ast}}}\underset{k\to+\infty}{\longrightarrow}+\infty.$$
  
Let us show now that $(5)$ implies $(6)$. Suppose that $(5)$ holds for some sequence $(n_k)_k$ and $v\in V$. Let $m\in\N$ be such that $v\in\Chi^{m}(\ro)$. Set $m_k=n_k+m$. Thus $\Chi^{n_k}(v)\subset \Chi^{m_k}(\ro)$ and hence
$$\sum_{u \in \Chi^{n_k}(v)}\dfrac{1}{|\mu_u|^{p^{\ast}}}\leqslant \sum_{u \in \Chi^{m_k}(\ro)}\dfrac{1}{|\mu_u|^{p^{\ast}}},$$
since the term at the left goes to infinity, as $k \to +\infty$,  we get
$$ \sum_{u \in \Chi^{m_k}(\ro)}\dfrac{1}{|\mu_u|^{p^{\ast}}} \underset{k\to+\infty}{\longrightarrow}+\infty.$$

Let us show that $(6)$ implies $(3)$.  Let $(n_k)_{k\geq0}$ be an increasing sequence of positive integers  such that
  $$\sum_{u \in \Chi^{n_k}(\ro)}\dfrac{1}{|\mu_u|^{p^{\ast}}}\underset{k\to+\infty}{\longrightarrow}+\infty.$$  
Set $J_k=\Chi^{n_k}(\ro)$, for $k\in\N$.  By Lemma \ref{Lemma_Grosse},  we obtain
$$\underset{\|x\|_1=1}{\inf}\Big(\sum_{u\in J_k}|x_u\mu_u|^p\Big)^{1/p}\underset{k\to+\infty}{\longrightarrow}0$$
 There exists then an increasing sequence $(k_j)_{j\geq0}$ of positive integers such that, for each $j\in\N_0$, we have
 $$\underset{\|x\|_1=1}{\inf}\Big(\sum_{u\in J_{k_j}}|x_u\mu_u|^p\Big)^{1/p}\leqslant\dfrac{1}{2^{j+1}c_j}, \quad\text{ where } \quad c_j=\max\lbrace 1; \|B^{n_{k_l}}\|, 0\leqslant l\leqslant j-1\rbrace,$$
 hence, there exists $f_{k_j}\in \K^{V}$, of support in $J_{k_j}=\Chi^{n_{k_j}}(\ro)$ such that
 \begin{equation}\label{eq6}
 \|f_{k_j}\|_1=\sum_{u\in \Chi^{n_{k_j}}(\ro)}|f_{k_j}(u)|=1,
 \end{equation}
 and
 \begin{equation}\label{eq7}
 \Big(\sum_{u\in J_{k_j}}|f_{k_j}(u)\mu_u|^p\Big)^{1/p}\leqslant\dfrac{1}{2^{j}c_j}.
 \end{equation}
 For each $u\in V$, set
 $$
 f(u)=\begin{cases}
 g_j(u) &\text{ if } u\in \Chi^{n_{k_j}}(\ro) \text{ for some } j\\
0 &\text{ otherwise }
\end{cases},\quad \text{where } g_j=|f_{k_j}|.
$$
By \eqref{eq7}, we obtain  
$$
\|f\|_{p,\mu}^{p}=\sum_{u\in V}|f(u)\mu_u|^p=\sum_{j\in\N_0}\sum_{u\in \Chi^{n_{k_j}}(\ro)}|f_{k_j}(u)\mu_u|^p \leqslant \sum_{j\in\N_0}\dfrac{1}{2^{pj}}<+\infty,
$$
thus $f\in \ell^p(V,\mu)$. Moreover, we have
\begin{align*}
\| B^{n_{k_j}}f-e_{\ro}\|_{p,\mu}&\leq \| B^{n_{k_j}}g_j-e_{\ro}\|_{p,\mu}+\sum_{ l=j+1} \| B^{n_{k_j}}g_l\|_{p,\mu}\\
&=\sum_{ l=j+1} \| B^{n_{k_j}}g_l\|_{p,\mu}\qquad (\text{by } \eqref{eq6})\\
&\leqslant \sum_{ l=j+1} \| B^{n_{k_j}}\| \|f_{k_l}\|_{p,\mu}\\
&< \sum_{ l=j+1}\dfrac{\| B^{n_{k_j}}\|}{2^{l}c_{l}} \qquad (\text{by } \eqref{eq7})\\
&\leqslant \sum_{ l=j+1}\dfrac{1}{2^{l}}\underset{j\to+\infty}{\longrightarrow}0, 
\end{align*}
hence condition $(3)$ holds.

 Let us show that $(5)$ implies $(7)$. Set $C=\max\lbrace 2, \|B\|\rbrace$. By assuming that $(5)$ holds, there exist  an increasing sequence $(n_{k})_{k\in\N}$ of positive integers and a vertex $v\in V$ such that
 \begin{equation}\label{eq1} 
 \Big(\sum_{u \in \Chi^{n_{k}+k}(v)}\dfrac{1}{|\mu_u|^{p^{\ast}}}\Big)^{-1/p^{\ast}}<C^{-2k}.
 \end{equation}
For every $u\in V\setminus\lbrace \ro\rbrace$, by Proposition \ref{Bounded_Back}, we obtain
$$C\geqslant   \Big(\sum_{v\in \Chi(\p(u))} \left|\frac{\mu_{\p(u)}}{\mu_v  }\right|^{p^\ast}\Big)^{1/p^{\ast}}\geqslant \left|\frac{\mu_{\p(u)}}{\mu_u }\right|,$$
that is, for every $u\in V\setminus\lbrace \ro\rbrace$,
\begin{equation}\label{eq2} 
\left|\frac{\mu_u }{\mu_{\p(u)}}\right|\geqslant C^{-1}.
\end{equation}
Fix now  $l\in\N_0$. Let $k\in\N$ be such that $k>l$.
Note that 
\begin{equation}\label{eq3} 
u\in \Chi^{n_{k}+k}(v) \Longleftrightarrow \p^{k-l}(u)\in \Chi^{n_{k}+l}(v),
\end{equation}
and for every $u\in \Chi^{n_{k}+k}(v)$, we have
$$\mu_{u}=\dfrac{\mu_u}{\mu_{\p(u)}}\times \dfrac{\mu_{\p(u)}}{\mu_{\p^2(u)}} \times \cdots\times \dfrac{\mu_{\p^{k-l-1}(u)}}{\mu_{\p^{k-l}(u)}}\times \mu_{\p^{k-l}(u)}.$$
By \eqref{eq2}, we obtain 
$$ \dfrac{1}{|\mu_{\p^{k-l}(u)}|^{p^{\ast}}}\geqslant C^{(l-k)p^{\ast}}\dfrac{1}{|\mu_{u}|^{p^{\ast}}}.$$
Combining this with \eqref{eq3}, we obtain
$$\Big(\sum_{u \in \Chi^{n_{k}+l}(v)}\dfrac{1}{|\mu_u|^{p^{\ast}}}\Big)^{1/p^{\ast}}\geqslant C^{l-k} \Big(\sum_{u \in \Chi^{n_{k}+k}(v)}\dfrac{1}{|\mu_u|^{p^{\ast}}}\Big)^{1/p^{\ast}}.$$
Thus
\begin{align*}
\Big(\sum_{u \in \Chi^{n_{k}+l}(v)}\dfrac{1}{|\mu_u|^{p^{\ast}}}\Big)^{-1/p^{\ast}}&\leqslant C^{k-l} \Big(\sum_{u \in \Chi^{n_{k}+k}(v)}\dfrac{1}{|\mu_u|^{p^{\ast}}}\Big)^{-1/p^{\ast}}\\
&\overset{\eqref{eq1}}{\leqslant} C^{k-l} C^{-2k}\leqslant C^{-k}  \underset{k\to+\infty}{\longrightarrow}0.
\end{align*}
 
\end{proof}

We now provide an example of a backward shift operator on a rooted directed tree that is not hypercyclic, but it has an orbit with a nonzero limit point.
\begin{example}\label{rooted_not_hc}
    Let $B$ be the backward shift operator given in Example \ref{back_not_hc}. It is clear that, by using Theorem \ref{FTransitive_rooted case}, $B$ is not hypercyclic. For every $N\in\mathbb{N}$, we have
     $$I(\ro,N):=\Big\{n\in\N_0:\, \Big(\sum_{u\in\Chi^n(\ro)} \dfrac{1}{|\mu_{u}|^{p^\ast}}\Big)^{1/p^\ast}>N\Big\}\in\mathcal{J},$$
    hence, by using Theorem \ref{recu_tree_shift}, we deduce that $B$ has an orbit with a non-zero limit point.
\end{example}

The following proposition provides a characterization of backward shift operators that have an orbit of a nonzero function with a nonzero limit point, specifically in the context of unrooted directed trees. We omit the proof as it can be derived through straightforward modifications of the arguments presented in the proofs of the preceding theorems.

\begin{proposition}\label{th_unrooted_tree}
Let $(V,E)$ be an unrooted directed tree and let $\mu=(\mu_v)_{v\in V}$ be a weight on $V$. Let $X=\ell^p(V,\mu)$, $1\leq p <+\infty$, or $X=c_0(V,\mu)$ and suppose  that the  backward shift $B$ is a bounded operator on $X$. The following conditions are equivalent:
\begin{enumerate}[label={$(\arabic*)$}]
\item  There exists a non-zero non-negative vector in $X$ whose orbit under $B$  has a nonzero limit point.
 \item There exist a  vertex $v\in V$ and a  non-zero non-negative vector in $X$ whose orbit under $B$ has $e_{v}$ as a limit point.
 \item  There exist an increasing sequence $(n_k)_{k}$ of positive integers and a vertex $v\in V$ such that
 $$\begin{cases}
\underset{u \in \Chi^{n_k}(v)}{\sum}\dfrac{1}{|\mu_u|^{p^{\ast}}}\underset{k\to+\infty}{\longrightarrow}+\infty & \text{ if } X=\ell^p(V,\mu),\, 1< p<+\infty;\\
 &\\
\underset{u\in\Chi^{n_{k}}(v)}{\inf}\,|\mu_u|\underset{k\to+\infty}{\longrightarrow}0 &\text{ if } X=\ell^1(V,\mu); \\
&\\
\underset{u \in \Chi^{n_k}(v)}{\sum}\dfrac{1}{|\mu_u|}\underset{k\to+\infty}{\longrightarrow}+\infty&\text{ if } X=c_0(V,\mu),
\end{cases}$$
  and, for every $i\in\N$, we have
$$\mu_{\p^{n_k-n_i}(v)}\underset{k\to+\infty}{\longrightarrow}0.$$ 
\end{enumerate}

\end{proposition}

In the following example, we provide a non-hypercyclic bounded backward shift on an unrooted directed tree that has an orbit with a nonzero limit point, but it does not possess an orbit of a non-negative function with a nonzero limit point.

\begin{example}\label{example1}

Let $(V,E)$ be the following unrooted directed tree:
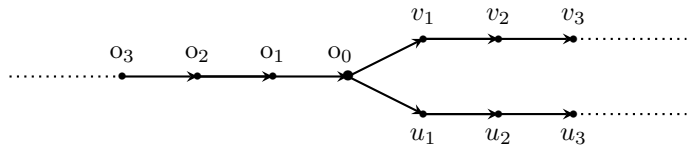
\begin{center}
\begin{tikzpicture}

\draw  (4.9,0.3) node {$\text{o}_0$};
\draw  (6,0.8) node {$v_1$};
\draw  (7,0.8) node {$v_2$};
\draw  (8,0.8) node {$v_3$};

\draw  (6,-0.8) node {$u_1$};
\draw  (7,-0.8) node {$u_2$};
\draw  (8,-0.8) node {$u_3$};

\draw  (4,0.3) node {$\text{o}_1$};
\draw  (3,0.3) node {$\text{o}_2$};
\draw  (2,0.3) node {$\text{o}_3$};

\draw (5,0) node {$\bullet$} ;
\draw (6,0.5) node {{\tiny$\bullet$}} ;

\draw (4,0) node {{\tiny$\bullet$}} ;
\draw (3,0) node {{\tiny$\bullet$}} ;
\draw (2,0) node {{\tiny$\bullet$}} ;

\draw (7,0.5) node {{\tiny$\bullet$}} ;
\draw (8,0.5) node {{\tiny$\bullet$}} ;

\draw (6,-0.5) node {{\tiny$\bullet$}} ;
\draw (7,-0.5) node {{\tiny$\bullet$}} ;
\draw (8,-0.5) node {{\tiny$\bullet$}} ;

\tikzstyle{suite}=[->,>=stealth,thick,rounded corners=4pt]


\draw[suite] (5,0) -- (6,0.5) edge (7,0.5) edge (8,0.5) ;
\draw[suite] (5,0) -- (6,-0.5) edge (7,-0.5) edge (8,-0.5);
\draw[suite] (2,0) -- (3,0) edge (4,0) edge (5,0);
\draw [dotted, thick] (8,0.5) -- (9.5,0.5);
\draw [dotted, thick] (8,-0.5) -- (9.5,-0.5);
\draw [dotted, thick] (0.5,0) -- (2,0);
\end{tikzpicture}
\end{center} 
that is
$$V:=\{u_k:\, k\in\N\}\cup \{v_k:\, k\in \N\}\cup \{ \text{o}_k:\, k\in\N_0\}.$$
Let $\mu=(\mu_v)_{v\in V}$ be the weight on $V$ defined by
$$\mu_{\text{o}_k}=1,\quad \mu_{u_k}=\dfrac{1}{2^k},\quad \text{and} \quad \mu_{v_k}=\dfrac{1}{2^k}.$$
Let $B$ be the backward shift on $\ell^p(V,\mu)$, $1<p<+\infty$. By Proposition \ref{Bounded_Back}, it is clear that $B$ is bounded and $\|B\|=2^{2-\frac{1}{p}}$. Now, consider any vertex $v\in V$, any increasing sequence $(n_k)_{k\in\N}$ of positive integers, and $i\in\N$, we have
$$\mu_{\p^{n_k-n_i}(v)}\underset{k\to+\infty}{\longrightarrow}1\neq0,$$
thus the  condition $(3)$ of Proposition \ref{th_unrooted_tree} does not hold. Consequently,  $B$ does not possess an orbit of a
non-negative function with a nonzero limit point. Moreover,  for any vertex $v\in V$ and any increasing sequence $(n_k)_{k\in\N}$ of positive integers, we have
$$\dfrac{1}{|\mu_{\p^{n_k}(v)}|^{p^\ast}}+\underset{u \in \Chi^{n_k}(\p^{n_k}(v))}{\sum}\dfrac{1}{|\mu_u|^{p^{\ast}}}\underset{k\to+\infty}{\centernot{\longrightarrow}}+\infty,$$
thus, by using   Theorem \ref{FTransitive_unrooted case} (when $\mathcal{F}$ is the Furstenberg family of infinite subsets of $\N_0$), we deduce that $B$ is not hypercyclic.

Let $f:V\longrightarrow\C$ be the function defined by
$$f(o_k)=0,\, f(u_{2^k})=1, \,f(v_{2^k})=-1,\,\text{and }  f(u_{j})=f(v_{j})=0,\text{ if } j\neq 2^k.$$
We have $f\in\ell^p(V,\mu)$, since
$$\|f\|_{p,\mu}^{p}=\sum_{v\in V}|f(v)|^p\mu_{v}^{p}=2\sum_{k\geq0}\dfrac{1}{2^{p2^k}}<+\infty.$$
Let us show that $B^{2^k-1}f\underset{k\to+\infty}{\longrightarrow}e_{u_1}-e_{v_1}$.
One has
\begin{align*}
 \| B^{2^k-1}f-e_{u_1}+e_{v_1}\|_{p,\mu}^p&=\sum_{j\geq0}|(B^{2^k-1}f)(o_j)|^p\mu_{o_j}^p+|(B^{2^k-1}f)(u_1)-1|^p\mu_{u_1}^{p}+ \sum_{j\geq2}|(B^{2^k-1}f)(u_j)|^p\mu_{u_j}^p\\
 &\qquad+|(B^{2^k-1}f)(v_1)+1|^p\mu_{v_1}^{p}+ \sum_{j\geq2}|(B^{2^k-1}f)(v_j)|^p\mu_{v_j}^p.
\end{align*}
We will show now that all these terms converge to $0$, as k goes to $+\infty$.
\begin{itemize}
    \item Let $k,j\in\N$. If $2^{k}-1\leq j$, then $\Chi^{2^k-1}(o_j)=\{o_{j-2^k+1}\}$, hence.
     $$(B^{2^k-1}f)(o_j)=f(o_{j-2^k+1})=0.$$
    If $2^{k}-1>j$, then $\Chi^{2^k-1}(o_j)=\{u_{2^k-1-j},v_{2^k-1-j}\}$, hence
    $$(B^{2^k-1}f)(o_j)=f(u_{2^k-1-j})+f(v_{2^k-1-j})=0.$$
    Thus
    $$\sum_{j\geq0}|(B^{2^k-1}f)(o_j)|^p\mu_{o_j}^p=0.$$
    \item Let $k\in\N$. We have
    $$(B^{2^k-1}f)(u_1)=\sum_{v\in \Chi^{2^{k}-1}(u_1)}f(v)=f(u_{2^k})=1,$$
    and
    $$ (B^{2^k-1}f)(v_1)=\sum_{v\in \Chi^{2^{k}-1}(v_1)}f(v)=f(v_{2^k})=-1.$$
    \item Let $k\in\N$. We have 
    \begin{align*}
        \sum_{j\geq2}|(B^{2^k-1}f)(u_j)|^p\mu_{u_j}^p&=\sum_{j\geq2}|f(u_{2^k-1+j})|^p\mu_{u_j}^p\\
        &=\sum_{l\geq k+1}|f(u_{2^l})|^p\mu_{u_{2^l-2^k+1}}^p\\
        &=\dfrac{1}{2^p}\sum_{l\geq k+1}\dfrac{1}{2^{p(2^l-2^k)}},
    \end{align*}
    since, for every $l\geq k+1$, $2^l-2^k\geq l$, we get
    $$\sum_{j\geq2}|(B^{2^k-1}f)(u_j)|^p\mu_{u_j}^p \leqslant \dfrac{1}{2}\sum_{l\geq k+1}\dfrac{1}{2^{l}}\underset{k\to+\infty}{\longrightarrow}0.$$
    \item Similarly, for each $k\in \N$, we have
    $$\sum_{j\geq2}|(B^{2^k-1}f)(v_j)|^p\mu_{v_j}^p=\dfrac{1}{2^p}\sum_{l\geq k+1}\dfrac{1}{2^{p(2^l-2^k)}}\underset{k\to+\infty}{\longrightarrow}0.$$
\end{itemize}
Hence
$$\| B^{2^k-1}f-e_{u_1}+e_{v_1}\|_{p,\mu}\underset{k\to+\infty}{\longrightarrow}0.$$
Therefore, $B$ has an orbit with a nonzero limit point. Note that the above computations also work when considering the backward shift operator $B$ as an operator acting on $\ell^1(V,\mu)$ or $c_0(V,\mu)$.
\end{example}

\addtocontents{toc}{\protect\setcounter{tocdepth}{0}}


\end{document}